\documentclass[a4paper,11pt]{amsproc}

\usepackage{xargs}
\usepackage{mathtools, todonotes}
\usepackage{amsmath,amssymb,amsfonts,amsthm}

\usepackage[colorlinks, backref]{hyperref}
\usepackage{amsrefs}
\usepackage{enumerate, tikz-cd, nicefrac}

\usepackage{breqn}

\title{Quadratic maps between non-abelian groups}

\author{Asgar Jamneshan}
\address{Asgar Jamneshan, University of Bonn, 53115 Bonn}
\email{ajamnesh@math.uni-bonn.de}

\author{Andreas Thom}
\address{Andreas Thom, TU Dresden, 01062 Dresden} 
\email{andreas.thom@tu-dresden.de}

\theoremstyle{plain}
\newtheorem{theorem}{Theorem}[section]
\newtheorem{definition}[theorem]{Definition}
\newtheorem{proposition}[theorem]{Proposition}
\newtheorem{lemma}[theorem]{Lemma}
\newtheorem{corollary}[theorem]{Corollary}

\theoremstyle{definition}
\newtheorem{remark}[theorem]{Remark}
\newtheorem{example}[theorem]{Example}

\newcommand{\beq}{\begin{equation}}
\newcommand{\eeq}{\end{equation}}
\newcommand{\beqn}{\begin{equation*}}
\newcommand{\eeqn}{\end{equation*}}
\newcommand{\brq}{\begin{dmath}[compact]}
\newcommand{\erq}{\end{dmath}}
\newcommand{\brqn}{\begin{dmath*}[compact]}
\newcommand{\erqn}{\end{dmath*}}
\newcommand{\bag}{\begin{align}}
\newcommand{\eag}{\end{align}}
\newcommand{\bagn}{\begin{align*}}
\newcommand{\eagn}{\end{align*}}

\newcommand{\vertiii}[1]{{|\kern-0.2ex|\kern-0.2ex| #1 
    |\kern-0.2ex|\kern-0.2ex|}}

\begin{document}
\begin{abstract}
Gowers and Hatami initiated the inverse theory for the uniformity norms $U^k$ of matrix-valued functions on non-abelian groups by proving a $1\%$-inverse theorem for the $U^2$-norm and relating it to stability questions for almost representations.  In this article, we take a step toward an inverse theory for higher-order uniformity norms of matrix-valued functions on arbitrary groups by examining the  $99\%$ regime for the $U^k$-norm on perfect groups of bounded commutator width.

This analysis prompts a classification of Leibman’s quadratic maps between non-abelian groups.  Our principal contribution is a complete description of these maps via an explicit universal construction.  From this classification we deduce several applications: A full classification of quadratic maps on arbitrary abelian groups; a proof that no nontrivial polynomial maps of degree greater than one exist on perfect groups; stability results for approximate polynomial maps. 
\end{abstract}

\maketitle

\tableofcontents

\section{Introduction}

Let $G$ be a finite group and let $k\ge1$ be an integer.  For any matrix-valued function $f\colon G\to M_n \mathbb C$, define the difference operator
$$
(\Delta_g f)(h)\;=\;f(gh)\,f(h)^*,
$$
and iterate to
$$
\Delta_{g_1,\dots,g_k}
=\;\Delta_{g_1}\circ\cdots\circ\Delta_{g_k}.
$$
The Gowers $U^k$-norm of $f$ is then given by
$$
\|f\|_{U^k}^{2^k}
\;=\;
\mathbb{E}_{g_0,\dots,g_k\in G}
\mathrm{tr}\bigl(\Delta_{g_1,\dots,g_k}f(g_0)\bigr),
$$
where $\mathrm{tr}$ denotes the normalized trace on $M_n \mathbb C$ and $\mathbb{E}$ the expectation with respect to the uniform distribution on $G^{k+1}$.  Basic properties of these matrix-valued Gowers norms - including the fact that $\|\cdot\|_{U^k}$ is indeed a norm - are collected in Appendix \ref{app:gowers}.

Higher-order Fourier analysis is built around the study of these Gowers uniformity norms and, in particular, the inverse problem: given a function $f$ with\footnote{For a function $\varphi \colon G \to M_n \mathbb{C}$, we write $\|\varphi\|_{\infty} := \sup_{g \in G} \|\varphi(g)\|$, where $\|\cdot\|$ denotes the usual operator norm of a matrix.} $\|f\|_\infty\le1$ and a large $U^k$-norm, what algebraic or analytic structure must $f$ exhibit?  This inverse question depends both on the degree $k$ of the norm and on the algebraic properties of the ambient group.  The theory of Gowers norms - and their inverse theorems - originated in Gowers’s Fourier-analytic proof of Szemerédi’s theorem for scalar-valued functions on the cyclic groups $\mathbb{Z}/n\mathbb{Z}$ \cite{g1}, and has since been developed chiefly for scalar-valued functions on finite abelian groups \cites{gtz,jt,leng,luka,tz1,tz2}.  Moreover, even in the abelian scalar-valued setting, the inverse theory remains far from complete (cf.\ \cite{jt}*{Conjecture 1.11}).

Gowers and Hatami~\cite{MR3733361} initiated the study of noncommutative Gowers norms by proving a quantitative inverse theorem for matrix-valued Gowers $U^2$-norms on nonabelian finite groups (see~\cite{MR3733361}*{Theorem 5.6}).  To illustrate, in the classical abelian scalar-valued setting, if $G$ is a finite abelian group, $f\colon G\to\mathbb{C}$ satisfies $\lvert f\rvert\le1$, and $\|f\|_{U^2}\ge\delta>0$, then Plancherel’s identity guarantees a character $\chi\colon G\to S^1$ with $|\langle f,\chi\rangle|\ge\delta^2$. In the noncommutative, matrix-valued case, characters are replaced by unitary representations, and the quantitative bounds depend on the dimensions of those representations as well.  

Gowers and Hatami further related their inverse theorem to stability questions for approximate representations in group theory.  Using operator-theoretic methods, these $U^2$-inverse and stability results were subsequently re-proved and extended to functions on countable amenable groups taking values in von Neumann algebras \cite{MR3867328}.

This article takes a step beyond the linear (first‐order) inverse Gowers theory by investigating the quadratic (second‐order) noncommutative Gowers inverse problem and its connection to higher‐order stability phenomena.  To initiate the study of the inverse theory for $\|f\|_{U^3}$ when $f\colon G\to M_n \mathbb C$ satisfies $\|f\|_\infty\le1$, it is natural to begin with the extremal case, namely $\|f\|_{U^3}=1$ ($100 \%$ regime).  Indeed, by definition,
$$
\|f\|_{U^3}^{8} \;=\; \mathbb{E}_{g_0,g_1,g_2,g_3} \mathrm{tr}\bigl(\Delta_{g_1,g_2,g_3}f(g_0)\bigr),
$$
so $\Delta_{g_1,g_2,g_3}f(g_0)=I_n$ for all $g_0,g_1,g_2,g_3\in G$ is equivalent to $\|f\|_{U^3}=1$.  Consequently, understanding the $100 \%$ regime amounts to classifying those functions $f\colon G\to M_n \mathbb C$ for which $\Delta_{g_1,g_2,g_3}f(g_0)=I_n$ for all $g_0,g_1,g_2,g_3\in G$. If we assume that $f$ takes unitary values, then such functions are precisely quadratic maps in the sense of Leibman, and, in the special case where $G$ is abelian and $f$ is scalar‐valued, they coincide with the non-classical phase polynomials that play a central role in the inverse theory of scalar‐valued Gowers norms on abelian groups \cites{jst,tz1,tz2}. 

Let $G$ and $H$ be groups. For a map $\varphi \colon G \to H$ and an element $k \in G$, define the \emph{finite difference} by
\[
(\Delta_k \varphi)(g) := \varphi(kg)\, \varphi(g)^{-1}.
\]

\begin{definition}\label{def-poly}
A map $\varphi \colon G \to H$ is said to be \emph{polynomial of degree~$-1$} if $\varphi(g) = 1$ for all $g \in G$.  
For $d \geq -1$, the map $\varphi$ is said to be \emph{polynomial of degree~$d+1$} if, for every $k \in G$, the finite difference $\Delta_k \varphi$ is polynomial of degree $d$.
\end{definition}

The notion of polynomial maps has emerged as a fundamental concept in the field of higher-order Fourier analysis. For example, complex functions of polynomial maps between filtered groups are believed to constitute the general obstructions to uniformity in the inverse theory of the Gowers norms over finite abelian groups~\cite{jt}. This belief has been confirmed in the setting of cyclic groups of prime order~\cite{gtz}, finite vector spaces~\cites{tz1,tz2}, and for the $U^3$-norm on general finite abelian groups~\cite{jt}. 

It is of interest to obtain algebraic descriptions and classifications of polynomial maps between groups. For classical polynomials over fields, a natural description is given by linear combinations of monomials, which can be derived from Definition~\ref{def-poly} using a Taylor expansion. In a similar vein, for polynomial maps from $\mathbb{Z}^d$ into a not necessarily commutative filtered group $G$, a kind of Taylor expansion is available; see, for instance,~\cite{gtz}*{Lemma B.9}. 

In the setting of maps from $k^n$, where $k$ is a finite field, into the torus $S^1$, a classification in terms of linear combinations of monomials divided by certain powers of the characteristic of $k$ was established in~\cite{tz2}*{Lemma 1.7}. Such maps, also known as \emph{non-classical polynomials}, take values in the group of roots of unity of order equal to a power of the characteristic, where the power depends on the degree. This contrasts with classical polynomials, whose values lie in the group of roots of unity of order equal to the characteristic itself.

However, even for polynomial maps from an arbitrary finite abelian group into the torus, no algebraic description or classification is currently known. The main result of this article presents a universal construction of polynomial maps between arbitrary groups and, in the case of quadratic maps, to compute this construction explicitly in group-theoretic terms. 

We now describe our results in more detail. It is  clear from the Definition \ref{def-poly} that the composition of a polynomial of degree $d$ with a homomorphism in either order is again a polynomial map of degree $d$. It also follows easily that polynomial maps of degree $0$ are constant functions, and that polynomial maps of degree $1$ are the product of a homomorphism and a constant function (see Lemma~\ref{lem:affine}). The primary focus of this paper is to study the first non-trivial case: quadratic maps, i.e., polynomials of degree at most two. It also follows directly from the definition that if $\varphi \colon G \to H$ is a quadratic map, then the map $\varphi' \colon G \to H$ defined by $\varphi'(g) = \varphi(g)h$, for some fixed $h \in H$, is also quadratic. Consequently, without loss of generality, we may assume that $\varphi(1) = 1$. We call a map $\varphi \colon G \to H$ \emph{unital} if it satisfies $\varphi(1) = 1$, and denote by ${\rm Quad}(G,H)$ the set of unital quadratic maps from $G$ to $H$. 
For a fixed group $G$, we can then interpret ${\rm Quad}(G,?)$ as a functor from the category of groups to the category of sets. Note that if $H$ is abelian, then ${\rm Quad}(G,H)$ is naturally an abelian group as well.

Our main result provides a characterization of quadratic maps in terms of group homomorphisms, as follows:

\begin{theorem} \label{thm:func}
Let $G$ be a group. There exists a group ${\rm Pol}_2(G)$ and a natural bijection 
\[
\hom({\rm Pol}_2(G),?) \to {\rm Quad}(G,?),
\]
i.e., the functor ${\rm Quad}(G,?)$ is representable. More specifically, there exists a universal unital quadratic map $\varphi \colon G \to {\rm Pol}_2(G)$ such that every unital quadratic map $\varphi' \colon G \to H$ can be expressed as $\varphi' = \psi \circ \varphi$ for a unique homomorphism $\psi \colon {\rm Pol}_2(G) \to H$.
\end{theorem}

We can describe the group ${\rm Pol}_2(G)$ in concrete terms as follows. Here, $\omega(G)$ denotes the augmentation ideal of the group ring $\mathbb{Z}[G]$, and $G^{ab}$ denotes the abelianization of $G$. 

\begin{theorem} \label{thm:pol2}
The group ${\rm Pol}_2(G)$ is isomorphic to the twisted crossed product group 
\[
(\omega(G) \otimes_{\mathbb{Z}} G^{ab}) \rtimes_{\psi} G,
\]
where $G$ acts on $\omega(G)$ by left multiplication and trivially on $G^{ab}$. The $2$-cocycle $\psi$ is given by $c \otimes {\rm ab}$, the exterior product of the tautological $1$-cocycle $c \colon G \to \omega(G)$ and the trivial $1$-cocycle ${\rm ab} \colon G \to G^{ab}$.
\end{theorem}

One may naturally inquire about the classification of polynomial maps of higher degree. While the definition of ${\rm Pol}_k(G)$ is straightforward in terms of generators and relations, obtaining an explicit description such as in Theorem \ref{thm:pol2} seems to be more challenging.

Theorem \ref{thm:pol2} has several noteworthy consequences. In particular, it provides a concrete description of the abelianization of ${\rm Pol}_2(G)$, where ${\rm Sym}^2(G)$ denotes the symmetric product of an abelian group $G$: 

\begin{corollary} \label{cor:abel}
Let $G$ be an abelian group.
The abelianization of ${\rm Pol}_2(G)$ fits into an extension
$$1 \to {\rm Sym}^2(G) \to {\rm Pol}_2^{ab}(G) \to G \to 1,$$
so that ${\rm Pol}_2^{ab}(G)$ is isomorphic to the twisted crossed product group
$${\rm Sym}^2(G) \rtimes_{\sigma} G$$ for the 2-cocycle $\sigma\colon G \times G \to {\rm Sym}^2(G)$  given by $\sigma(g,h)=gh$. 
\end{corollary}

We establish Theorem \ref{thm:pol2} and Corollary \ref{cor:abel} in Section \ref{sec:proof}. 

We begin our discussion of the consequences of our main results with the case of non-commutative groups.
Recall that a group $G$ is called \emph{perfect} if it coincides with its commutator subgroup $[G, G]$, or equivalently, if its abelianization $G^{\mathrm{ab}}$ is trivial. 

\begin{corollary}\label{cor:perfect}
If $G$ is a perfect group, then the natural homomorphism $\pi \colon {\rm Pol}_2(G) \to G$ is an isomorphism. Consequently, every unital polynomial map defined on a perfect group is a group homomorphism.
\end{corollary}

Corollary \ref{cor:perfect} is proved in Section~\ref{sec:cons}, where we also compute ${\rm Pol}_2(\mathbb{F}_d)$ and its abelianization for the free group $\mathbb{F}_d$ on $d$ generators, and show that ${\rm Pol}_2(G)$ is finite for every finite group $G$.

In Section~\ref{sec:abel}, we turn our attention to the case of abelian groups. Of particular interest is the computation of the group of unital quadratic maps ${\rm Quad}(G, S^1)$ for an arbitrary finite abelian group $G$. In Proposition~\ref{prop:primary decomposition}, we establish that
\[
{\rm Quad}(G, S^1) = \bigoplus_p {\rm Quad}(G_{(p)}, S^1),
\]
as abelian groups, where $\bigoplus_p G_{(p)}$ denotes the primary decomposition of $G$. Proposition~\ref{classification} below then gives a complete classification of ${\rm Quad}(G, S^1)$ for an arbitrary finite abelian group $G$. 

\begin{proposition}\label{classification}
Let $p$ be a prime number, and let $G = \bigoplus_{i=1}^d \mathbb{Z}/p^{n_i}\mathbb{Z}$ be a finite abelian $p$-primary group. 
\begin{itemize}
\item[(i)] If $p = 2$, then there is a group isomorphism
\[
{\rm Quad}(G, S^1) \cong \bigoplus_{i=1}^d \mathbb{Z}/2^{n_i - 1}\mathbb{Z} \oplus \bigoplus_{1 \leq i < j \leq d} \mathbb{Z}/2^{\min(n_i, n_j)}\mathbb{Z} \oplus \bigoplus_{i=1}^d \mathbb{Z}/2^{n_i + 1}\mathbb{Z}.
\]
\item[(ii)] If $p \neq 2$, then there is a group isomorphism
\[
{\rm Quad}(G, S^1) \cong \bigoplus_{i=1}^d \mathbb{Z}/p^{n_i}\mathbb{Z} \oplus \bigoplus_{1 \leq i < j \leq d} \mathbb{Z}/p^{\min(n_i, n_j)}\mathbb{Z} \oplus \bigoplus_{i=1}^d \mathbb{Z}/p^{n_i}\mathbb{Z}.
\]
\end{itemize}
\end{proposition}

Our results for abelian groups recover the Tao--Ziegler classification for non-classical polynomials on finite vector spaces~\cite{tz2}*{Lemma~1.7} in the quadratic case (see Remark \ref{tz}), as well as the aforementioned Taylor expansion formula for quadratic maps on~$\mathbb{Z}^d$ (see Example \ref{taylor}). Moreover, from the inverse theorem for the Gowers $U^3$-norm on arbitrary finite abelian groups~\cite{jt}, one can deduce that, when restricting to families of finite abelian groups with uniformly bounded torsion, the obstructions to uniformity are precisely the quadratic maps with values in~$S^1$. In this setting, Proposition~\ref{classification} provides an explicit description of the obstruction.

The final two sections are devoted to applications concerning the stability of approximate polynomial maps and a non-commutative inverse Gowers theory. Recall that a group $G$ is said to have \emph{commutator width} $k$ if every element $g \in G$ can be expressed as a product of at most $k$ commutators. The results in these sections focus on perfect groups with bounded commutator width. In both applications, Theorem~\ref{thm:pol2} is used in the form of Corollary~\ref{cor:perfect}.

In Section~\ref{sec:apppol}, we study approximate polynomial maps with values in metric groups. Let $(H, \partial)$ be a group equipped with a bi-invariant metric. A unital map $\varphi \colon G \to H$ is called a \emph{uniform $\varepsilon$-polynomial of degree~$d$} if
\[
\partial((\Delta_{g_1, \dots, g_{d+1}} \varphi)(1), 1) \leq \varepsilon
\]
for all $g_1, \dots, g_{d+1} \in G$, where $\Delta_{g_1, \dots, g_k} := \Delta_{g_1} \circ \cdots \circ \Delta_{g_k}$. A uniform $\varepsilon$-polynomial of degree one is called uniform $\varepsilon$-homomorphism -- in agreement with the usual understanding of this term, since
$$d((\Delta_{g_1,g_2} \varphi)(1),1) = d(\varphi(g_2g_1)\varphi(g_1)^* \varphi(g_2)^*,1) = d(\varphi(g_2g_1),\varphi(g_2)\varphi(g_1)).$$ We establish the following strong stability result for such approximate polynomial maps.

\begin{theorem} \label{thm:approx}
Let $k \in \mathbb{N}$. There exists a constant $c := c(k) > 0$ such that the following holds:  
Let $d \in \mathbb{N}$, let $(H, \partial)$ be a group equipped with a bi-invariant metric, and let $G$ be a group of commutator width at most $k$. Then every unital uniform $\varepsilon$-polynomial map $\varphi \colon G \to H$ of degree~$d$ is a uniform $c(k)^{d-1} \varepsilon$-homomorphism.
\end{theorem}

In Section~\ref{sec:gowers}, inspired by the work of Gowers--Hatami~\cite{MR3733361}, we define Gowers norms of arbitrary degree for functions on a finite group with values in $M_n \mathbb{C}$ as above. Any polynomial map $\varphi \colon G \to \mathrm{U}(n)$ of degree $k-1$ satisfies $\|\varphi\|_{U^k} = 1$. It is thus natural to ask whether functions with values in the unitary group and $U^k$-norm close to $1$ must be close, in a suitable sense, to a polynomial map - where the estimates depend only on $k$ and the commutator width, and are independent of the group $G$ and the dimension $n$. This includes, for example, uniform estimates for the class of non-abelian finite simple groups, whose commutator width is known to be equal to one, see \cite{ore} and Remark \ref{rem:simple}.

Such a statement is often referred to as an \emph{inverse theorem in the $99\%$ regime}. Our main result is a strong form of such a theorem, which entails a collapse of the higher-order Fourier hierarchy: 

\begin{theorem} \label{thm:gowers}
Let $k_0, k \in \mathbb{N}$. For every $\delta > 0$ there exists $\varepsilon = \varepsilon(\delta, k, k_0) > 0$ such that the following holds:  
Let $G$ be a finite perfect group of commutator width at most $k_0$, and let $\varphi \colon G \to M_n \mathbb{C}$ be a map satisfying $\|\varphi\|_{\infty} \leq 1$ and
\[
\|\varphi\|^{2^k}_{U^k} \geq 1 - \varepsilon.
\]
Then there exists a homomorphism $\beta \colon G \to \mathrm{U}(n')$ with $n' \in [n, (1+\delta)n] \cap \mathbb{N}$ such that
\[
\mathbb{E}_{g \in G} \left\| (\varphi(1)^* \varphi(g) \oplus 1_{n'-n}) - \beta(g) \right\|_2^2 < \delta.
\]
\end{theorem}

By Corollary \ref{cor:perfect}, every unital polynomial map is necessarily a homomorphism.  Hence Theorem \ref{thm:gowers} can be viewed as extending the Gowers--Hatami inverse theorem (\cite{MR3733361}*{Theorem 5.6}) from the case $k=2$ to all higher degrees - albeit only in the $99\%$ regime (whereas Gowers--Hatami treat the more demanding $1\%$ regime). 
Such inverse theorems require strong structural hypotheses on the group; our proof assumes that $G$ is perfect.  Whether a similar collapse of the higher‐order Fourier hierarchy persists in the $1\%$ regime for suitable classes of groups remains an intriguing open question, and at present no further evidence or results are known.

\section{Basic results}

Before we get to quadratic maps, we recall a basic result on polynomial maps of degree one.

\begin{lemma} \label{lem:affine}
A map $\varphi \colon G \to H$ is polynomial of degree one if and only if there exists $c \in H$ and a homomorphism $\psi \colon G \to H$, such that
$$\varphi(g) = \psi(g)c, \quad \forall g \in G.$$
\end{lemma}
\begin{proof}
We set $\psi(g):=\varphi(g)\varphi(1)^{-1}$. 
Note that $(\Delta_g \varphi)(h) = \varphi(gh)\varphi(h)^{-1}$ is of degree zero and therefore a constant equal to $\varphi(g)\varphi(1)^{-1}$. Thus, we obtain
$$\psi(gh) = \varphi(gh) \varphi(1)^{-1} = \varphi(g) \varphi(1)^{-1} \varphi(h) \varphi(1)^{-1} = \psi(g) \psi(h)$$
and conclude that $\psi$ is a homomorphism, and set $c:=\varphi(1)$. 
\end{proof}

We will freely use the terminology of group actions and cocycles, see \cite{MR1324339} for details and background.
Note that there is a natural right action of $G$ on the set of functions $f$ from $G$ to $H$, given by $f^{k}(g)=f(kg)$ for all $k,g \in G$. Indeed, we compute \[f^{k_1k_2}(g)=f(k_1k_2g) = f^{k_1}(k_2g) = (f^{k_1})^{k_2}(g).\] In terms of this action, the finite differences satisfy the following cocycle relation. 

\begin{lemma}
Let $\varphi \colon G \to H$ be a map and $k_1,k_2 \in G$, then the following $1$-cocycle identity holds:
\begin{equation} \label{eq:coc}
\Delta_{k_1k_2} \varphi = (\Delta_{k_1} \varphi)^{k_2} \cdot \Delta_{k_2} \varphi
\end{equation}
\end{lemma}
\begin{proof}
We compute \begin{eqnarray*}
(\Delta_{k_1k_2} \varphi)(g)&=&\varphi(k_1k_2g)\varphi(g)^{-1} \\
&=& \varphi(k_1k_2g)\varphi(k_2g)^{-1}\varphi(k_2g)\varphi(g)^{-1} \\
&=& (\Delta_{k_1} \varphi)(k_2g) \cdot (\Delta_{k_2}\varphi)(g)
\end{eqnarray*}
\end{proof}

Let $\varphi \colon G \to H$ be a unital quadratic map. Using Lemma \ref{lem:affine}, for each $k \in G$ there exists a homomorphism $\psi_k \colon G \to H$ and a constant $c_k \in H$, such that
$$(\Delta_k \varphi)(g)=\varphi(kg)\varphi(g)^{-1} = \psi_k(g)c_k, \quad \forall g,k \in G.$$
For $g=1$, we obtain $c_k=\varphi(k).$ 

Setting
$$\beta_k(g):= \varphi(k)^{-1} \psi_k(g) \varphi(k) = \varphi(k)^{-1} (\Delta_k \varphi)(g)$$ we obtain from Equation\eqref{eq:coc} the following identity:
\begin{equation} \label{eq:cocycbeta}
\beta_{k_1k_2}(g) = \varphi(k_2)^{-1}\beta_{k_1}(g)\varphi(k_2) \cdot \beta_{k_2}(g).
\end{equation}
Indeed, we compute as follows:
\begin{eqnarray*}
\beta_{k_1k_2}(g) &=& \varphi(k_1k_2)^{-1} (\Delta_{k_1k_2} \varphi)(g)    \\
&=& \varphi(k_1k_2)^{-1} (\Delta_{k_1} \varphi)^{k_2}(g)  (\Delta_{k_2} \varphi)(g) \\
&=& \varphi(k_1k_2)^{-1} (\Delta_{k_1} \varphi)(k_2g)  (\Delta_{k_2} \varphi)(g) \\
&=& \varphi(k_1k_2)^{-1} \varphi(k_1) \beta_{k_1}(k_2g) \varphi(k_2) \beta_{k_2}(g) \\
&=& \varphi(k_1k_2)^{-1} \varphi(k_1) \beta_{k_1}(k_2) \beta_{k_1}(g) \varphi(k_2) \beta_{k_2}(g) \\
&=& \varphi(k_2)^{-1} \beta_{k_1}(g) \varphi(k_2) \cdot \beta_{k_2}(g).
\end{eqnarray*}

Note that $\beta_k \colon G \to H$ is also a homomorphism for $k \in G$ and $\beta_1$ is trivial. Moreover, we conclude from Equation \eqref{eq:cocycbeta} that the pointwise product of the homomorphism $g \mapsto \varphi(k_2)^{-1}\beta_{k_1}(g)\varphi(k_2)$ and the homomorphism $g \mapsto \beta_{k_2}(g)$ is again a homomorphism. The following folklore lemma shows that this is a rather special situation.

\begin{lemma} \label{lem:comm}
Let $\alpha,\beta \colon G \to H$ be homomorphisms. \begin{enumerate}[$(i)$]
\item If the pointwise product $\alpha \beta$ is again a homomorphism, then the ranges of $\alpha$ and $\beta$ commute. \item If the pointwise inverse $\alpha^{-1}$ is again a homomorphism, then the range of $\alpha$ is abelian.
\end{enumerate}
\end{lemma}
\begin{proof}
We first prove Claim $(i)$ and compute:
\begin{eqnarray*}
\alpha(g)\alpha(h)\beta(g) \beta(h) &=& \alpha(gh)\beta(gh) = (\alpha\beta)(gh) \\&=& (\alpha \beta)(g)(\alpha \beta)(h) =\alpha(g)\beta(g)\alpha(h)\beta(h).
\end{eqnarray*}
Hence, $\alpha(h)\beta(g)=\beta(g)\alpha(h)$ for all $g,h \in G.$ Claim $(ii)$ follows since the pointwise product of $\alpha$ and $\alpha^{-1}$ is the trivial homomorphism and the ranges of $\alpha$ and $\alpha^{-1}$ agree.
\end{proof}

\begin{proposition} \label{prop:comm} The range of $\beta_k$ is abelian for all $k \in G$ and the ranges of the homomorphisms $\beta_{k_1}$ and $\beta_{k_2}$ commute for all $k_1,k_2 \in G.$ 
\end{proposition}
\begin{proof}
Indeed, Equation \eqref{eq:cocycbeta} shows for $k_1=k,k_2=k^{-1}$ that the map $g \mapsto \beta_k(g)^{-1}$ is also a homomorphism. By Lemma \ref{lem:comm}(ii), this implies that the range of $\beta_k$ is abelian for all $k \in G$. Now, Lemma \ref{lem:comm}(i) implies that the ranges of $g \mapsto \varphi(k_2)^{-1}\beta_{k_1}(g)\varphi(k_2)$ and $\beta_{k_2}$ commute. Hence, we conclude also that their product $\beta_{k_1k_2}$ commutes with $\beta_{k_2}$. Since $k_1,k_2$ were arbitrary, this proves the claim.
\end{proof}

We can already prove our first result.

\begin{proof}[Proof of Theorem \ref{thm:func}:] Let $G$ be a group.
We define a group ${\rm Pol}_2(G)$ on formal generators $\varphi(g)$ for $g \in G$, subject to relations
\begin{equation} \label{eq:rel}
\varphi(g_1g_2g_3)= \varphi(g_1g_2)\varphi(g_2)^{-1}\varphi(g_1)^{-1} \varphi(g_1g_3)\varphi(g_3)^{-1}\varphi(g_2g_3)
\end{equation}
for all $g_1,g_2,g_3 \in G.$ We claim that the tautological map $\varphi \colon G \to {\rm Pol}_2(G)$ is the universal unital quadratic map. Applying the relation to $(g_1,g_2,g_3)=(1,1,1)$, we see that $\varphi(1)=1$, so that $\varphi$ is unital.

The requirement for $\varphi$ to be quadratic is precisely that $\Delta_{g_1} \varphi$ is linear for every $g_1$. Recall that $$(\Delta_{g_1} \varphi)(g) := \varphi(g_1g)\varphi(g)^{-1},$$ so that $(\Delta_{g_1} \varphi)(1)= \varphi(g_1).$ We set
$$\psi_{g_1}(g) := \varphi(g_1g)\varphi(g)^{-1} \varphi(g_1)^{-1}.$$
It remains to show that $\psi_{g_1}$ is a homomorphism for every $g_1.$ However,
\begin{eqnarray*}
\psi_{g_1}(g_2)\psi_{g_1}(g_3) &=& \varphi(g_1g_2)\varphi(g_2)^{-1} \varphi(g_1)^{-1}\varphi(g_1g_3)\varphi(g_3)^{-1} \varphi(g_1)^{-1} \\
&\stackrel{(\ref{eq:rel})}{=}& \varphi(g_1g_2g_3) \varphi(g_2g_3)^{-1} \varphi(g_1)^{-1} \\
&=& \psi_{g_1}(g_2g_3).
\end{eqnarray*}
This proves the claim and shows that $\varphi$ is indeed quadratic. At the same time, we see that all relations are necessary so that it follows that $\varphi$ is indeed the universal quadratic map.
\end{proof}

\begin{remark} It is clear form the construction that a similar approach would define ${\rm Pol}_k(G)$, the relations to put are just $((\Delta_{g_1} \circ \cdots \circ \Delta_{g_{k+1}}) \varphi)(g_{k+2})$ for all $g_1,\dots,g_{k+2} \in G$ and $\varphi(1)=1.$
\end{remark}

\begin{proposition}
The assignment $G \mapsto {\rm Pol}_2(G)$ is a functor on the category of groups.
\end{proposition}

The following corollary will allow us to compute ${\rm Pol}_2(G)$ in various special cases.

\begin{corollary} \label{cor:comp}
Let $G$ be a group with symmetric generating set $S \subseteq G$ and let $N \leq G$ be a normal subgroup generated by a symmetric set $\Sigma \subseteq N$. Then ${\rm Pol}_2(G/N)$ is canonically isomorphic to the quotient of ${\rm Pol}_2(G)$ by the normal subgroup generated by the elements $\varphi(\sigma)$ and $\varphi(\sigma s) \varphi(s)^{-1}$ for $\sigma \in \Sigma, s \in S.$\end{corollary}
\begin{proof}
We set $K_N := \langle \! \langle \varphi(\sigma), \varphi(\sigma s) \varphi(s)^{-1}\colon  \sigma \in \Sigma, s \in S \rangle \! \rangle$ and aim to show that ${\rm Pol}_2(G)/K_N$ has the required universal property. Note that $K_N$ lies in the kernel of the canonical map ${\rm Pol}_2(G) \to {\rm Pol}_2(G/N).$ It remains to provide an inverse to the induced homomorphism ${\rm Pol}_2(G)/K_N \to {\rm Pol}_2(G/N).$ In order to do so we show that the canonical quadratic map $\bar \varphi \colon G \to {\rm Pol}_2(G)/K_N$ is well-defined on $G/N.$

Now, it follows that $\beta_{\sigma}(s) = \varphi(\sigma)^{-1}\varphi(\sigma s) \varphi(s)^{-1}  \in K_N$ for $\sigma \in \Sigma, s\in S$ and hence $\beta_{\sigma}(g) \in K_N$ for all $g \in G$. We conclude from Equation \eqref{eq:cocycbeta} and the fact that $K_N$ is normal, that $\beta_h(g) \in K_N$ for all $h \in N$ and $g \in G.$ Since $\beta_{h}(h')= \varphi(h)^{-1} \varphi(hh')\varphi(h')^{-1}$ it follows by induction on the word length that $\varphi(h) \in K_N$ for all $h \in N.$ Finally, this implies that $\varphi(hg)\varphi(g)^{-1} \in K_N$ for all $h \in N,g \in G$. Thus, $\bar \varphi$ is indeed well-defined on $G/N.$
\end{proof}

\section{The computation of ${\rm Pol}_2(G)$ and its abelianization} \label{sec:proof}

This section is concerned with the proof of our main result about the structure of ${\rm Pol}_2(G).$ This is the basis for later sections, in which we exploit this structure result in order to derive explicit computations of ${\rm Pol}_2(G)$ in special cases and other consequences.

\begin{proof}[Proof of Theorem \ref{thm:pol2}]
We have an extension
$$1 \to {\rm N}(G) \to {\rm Pol}_2(G) \stackrel{\pi}{\to} G \to 1$$
where $\pi$ is the homomorphism that classifies the unital quadratic map ${\rm id}\colon G \to G$, and the group ${\rm N}(G)$ is defined to be $\ker(\pi)$. Since $\pi(\varphi(g))=g$, we have that $\varphi \colon G \to {\rm Pol}_2(G)$ is a canonical set-theoretic section of $\pi$.  Recall that $\beta_k(g) = \varphi(k)^{-1}\varphi(kg) \varphi(g)^{-1}$ for all $k,g \in G$. 
\begin{proposition} 
The group ${\rm N}(G)$ is abelian and generated as a group by the set $\{\beta_k(g)\colon  g,k \in G \setminus \{1\}\}.$
\end{proposition}
\begin{proof}
We obtain from Equation \eqref{eq:cocycbeta} that
$$\varphi(k_2)^{-1} \beta_{k_1}(g) \varphi(k_2)= \beta_{k_1k_2}(g) \beta_{k_2}(g)^{-1}.$$
Since ${\rm Pol}_2(G)$ is generated by $\varphi(G)$, we conclude that the group $\langle \beta_k(g)\colon g,k \in G \setminus \{1\} \rangle$ is normal in $G$. 
Now, since the relations of the form $\beta_k(g)$ define the quotient $G={\rm Pol}_2(G)/{\rm N}(G)$, we obtain ${\rm N}(G)=\langle \beta_k(g)\colon g,k \in G \setminus \{1\} \rangle$ as claimed. By Proposition \ref{prop:comm}, ${\rm N}(G)$ is abelian.
\end{proof}

Let's try to analyze the extension 
$$1 \to {\rm N}(G) \to {\rm Pol}_2(G) \to G \to 1$$ more closely. The study of extensions with abelian kernel is a classical topic, see, e.g., \cite{MR1324339}*{p.90} for background.

First observe, that there is a homomorphism $\mu \colon G \to {\rm Aut}({\rm N}(G)),$ given by the formula
$\mu_g(h) = \varphi(g)h\varphi(g)^{-1}$ for $h \in {\rm N}(G), g \in G$. Indeed, this is a homomorphism since $\varphi(g)^{-1}\varphi(gh)\varphi(h)^{-1} =
\beta_g(h) \in {\rm N}(G)$ and ${\rm N}(G)$ is abelian.

Now, if we write the cocycle relation in additive notation, we obtain:
\begin{equation} \label{eq:additcoc}
\beta_{k_1k_2}(g) = \mu_{k_2^{-1}}(\beta_{k_1}(g)) + \beta_{k_2}(g).
\end{equation}
or equivalently
\begin{equation*}
\mu_{k_2^{-1}}(\beta_{k_1}(g)) = \beta_{k_1k_2}(g) - \beta_{k_2}(g).
\end{equation*}
There is also a 2-cocycle $\psi \colon G \times G \to {\rm N}(G)$ with respect to the action just defined. The formula for $\psi$ is given by
$$\psi(g,h) = - \mu_g(\beta_{g}(h)) = \varphi(g)\varphi(h)\varphi(gh)^{-1}.$$

Recall, that in formulas being a 2-cocycle means:
$$\mu_{g_1}(\psi(g_2,g_3)) - \psi(g_1g_2,g_3) + \psi(g_1,g_2g_3) - \psi(g_1,g_2)=0$$
for all $g_1,g_2,g_3 \in G,$ which is easily checked.
Indeed, for sake of completeness, we compute:
\begin{eqnarray*}
&& \mu_{g_1}(\psi(g_2,g_3)) - \psi(g_1g_2,g_3) + \psi(g_1,g_2g_3) - \psi(g_1,g_2) \\
&=& - \mu_{g_1}(\mu_{g_2}(\beta_{g_2}(g_3))) + \mu_{g_1g_2}(\beta_{g_1g_2}(g_3))- \mu_{g_1}(\beta_{g_1}(g_2g_3)) + \mu_{g_1}(\beta_{g_1}(g_2)) \\
&=& \mu_{g_1}(- \mu_{g_2}(\beta_{g_2}(g_3)) + \mu_{g_2}(\beta_{g_1g_2}(g_3))- \beta_{g_1}(g_2g_3) + \beta_{g_1}(g_2)) \\
&=& \mu_{g_1}(- \mu_{g_2}(\beta_{g_2}(g_3)) + \mu_{g_2}(\beta_{g_1g_2}(g_3))- \beta_{g_1}(g_3)) \\
&=& \mu_{g_1g_2}(- \beta_{g_2}(g_3) + \beta_{g_1g_2}(g_3)- \mu_{g_2^{-1}}(\beta_{g_1}(g_3))) \stackrel{(\ref{eq:additcoc})}{=} 0.
\end{eqnarray*}

The extension would be described completely as $${\rm Pol}_2(G)= {\rm N}(G) \rtimes_{\mu,\psi} G,$$ after determining the $\mathbb Z[G]$-module ${\rm N}(G)$ and understanding the cohomology class $[\psi] \in H^2(G,{\rm N}(G))$ more explicitly. Recall that the multiplication in the twisted crossed product is given by
$$(\xi_1,g_1)(\xi_2,g_2) = (\xi_1 + \mu_{g_1}(\xi_2) + \psi(g_1,g_2),g_1g_2).$$

We denote by $\omega(G)$ the augmentation ideal of $\mathbb Z[G]$, i.e., $\omega(G) := \{ a \in \mathbb Z[G], \varepsilon(a)=0\}$ with $\varepsilon(\sum_g a_g g):= \sum_g a_g.$
There is a natural free basis for $\omega(G)$ consisting of the set $\{g-1 \mid g \in G, g \neq 1\}.$ We set $c(g):=g-1$ and note that $gc(h) = gh-g = c(gh) - c(g).$ This is the tautological $1$-cocycle $c \colon G \to \omega(G).$

The key to understand ${\rm N}(G)$ is to note that there exists a $G$-equivariant surjection
$\alpha \colon  \omega(G) \otimes_{\mathbb Z }G^{ab}  \to {\rm N}(G)$ given by the formula
$\alpha(c(h) \otimes \bar g ) = \beta_{h^{-1}}(g),$ where $\bar g$ denotes the image of $g \in G$ in $G^{ab}$. The $G$-equivariance is checked by the following computation:
\begin{eqnarray*}
\alpha(g_1(c(h) \otimes \bar g))&=& \alpha(c(g_1h) \otimes \bar g)) - \alpha(c(g_1) \otimes \bar g) \\
&=& \beta_{h^{-1}g_1^{-1}}(g) - \beta_{g_1^{-1}}(g)\\
&=& \mu_{g_1}(\beta_{h^{-1}}(g))\\
&=& \mu_{g_1}(\alpha(c(h) \otimes \bar g)).
\end{eqnarray*}
We claim that $\alpha$ is an isomorphism. In order to show this, we would like to identify a 2-cocycle $\psi' \colon G\times G \to \omega(G) \otimes_{\mathbb Z} G^{ab}$ compatible with $\alpha$ and $\psi$, so that the natural map
$$(\omega(G) \otimes_{\mathbb Z} G^{ab}) \rtimes_{\psi'} G \to {\rm Pol}_2(G)$$
is an isomorphism. This then also gives the desired computation of ${\rm Pol}_2(G).$

First of all, we need a good formula for the $2$-cocycle $\psi' \colon G \times G \to \omega(G) \otimes_{\mathbb Z }G^{ab}.$
Since we have
$\psi(g,h)=\varphi(g)\varphi(h)\varphi(gh)^{-1}
= \mu_g(\beta_g(h))^{-1}$
a natural guess for a $2$-cocycle $\psi' \colon G \times G \to \omega(G) \otimes G^{ab}$ would be
$$\psi'(g,h)=- g (c(g^{-1}) \otimes \bar h) =  c(g) \otimes \bar h.$$ This is indeed a 2-cocycle, namely the exterior product of the tautological $1$-cocycle $c \colon G \to \omega(G)$ and the trivial $1$-cocycle ${\rm ab} \colon G \to G^{ab}$, where $G$ acts trivially on the abelian group $G^{ab}.$
Moreover, if we compose the $2$-cocycle $\psi' \colon G \times G \to \omega(G) \otimes_{\mathbb Z }G^{ab}$ with the surjective $G$-module homomorphism $\alpha \colon  \omega(G) \otimes_{\mathbb Z }G^{ab}  \to {\rm N}(G)$, then we obtain $\psi$ by construction.

Thus, we may consider the twisted crossed product group
$(\omega(G) \otimes_{\mathbb Z} G^{ab}) \rtimes_{\psi'} G$ with its natural multiplication
$$(\xi_1,g_1)(\xi_2,g_2) := (\xi_1 + g_1(\xi_2) + \psi'(g_1,g_2),g_1g_2).$$
We obtain a natural surjection $$\alpha' \colon (\omega(G) \otimes_{\mathbb Z} G^{ab}) \rtimes_{\psi'} G \to {\rm Pol}_2(G)$$ given by
$\alpha'((c(h) \otimes \bar g_1),g_2) = \beta_{h^{-1}}(g_1) \varphi(g_2).$

In order to find the inverse of $\alpha'$, we use the universal property of ${\rm Pol}_2(G)$ and just construct a quadratic map
$$\varphi' \colon G \to (\omega(G) \otimes_{\mathbb Z} G^{ab}) \rtimes_{\psi'} G.$$ Our philosophy suggests that it must be 
$\varphi'(g):= (0,g)$ in the notation above. Indeed, if $\varphi'$ is quadratic, it is easy to see that its classifying map is the inverse of $\alpha$. The inverse of $\varphi'(g)$ is naturally computed as $\varphi'(g)^{-1} = (-\psi'(g^{-1},g),g^{-1}).$

Now, for any $g \in G$, we obtain
\begin{eqnarray*}
\beta'_g(h)&=& \varphi'(g)^{-1} \varphi'(gh) \varphi'(h)^{-1}\\
&=& (-\psi'(g^{-1},g),g^{-1})(0,gh)(-\psi'(h^{-1},h),h^{-1}) \\
&=& (-\psi'(g^{-1},g) + \psi'(g^{-1},gh),h)(-\psi'(h^{-1},h),h^{-1})\\
&=& (-\psi'(g^{-1},g) + \psi'(g^{-1},gh) - h \psi'(h^{-1},h) + \psi(h,h^{-1}),1) \\
&=& -c(g^{-1}) \otimes \bar g + c(g^{-1}) \otimes (\bar g +\bar h) - hc(h^{-1}) \otimes \bar h - c(h) \otimes h\\
&=& c(g^{-1}) \otimes \bar h.
\end{eqnarray*}
This is clearly a homomorphism in $h$. Thus, $\varphi'$ is quadratic and the proof of Theorem \ref{thm:pol2} is complete.
\end{proof}

Having Theorem \ref{thm:pol2} at hand it is not difficult to describe its abelianization in concrete terms.

\begin{proof}[Proof of Corollary \ref{cor:abel}] Let $G$ be an abelian group.
We need to compute the abelianization of the group $(\omega(G) \otimes_{\mathbb Z} G) \rtimes_{\psi} G$, where $
\psi \colon G \times G \to \omega(G) \otimes_{\mathbb Z} G^{ab}$ is given by $\psi(g,h)=c(g) \otimes \bar h.$ We have
$$(\xi_1,g_1)(\xi_2,g_2) := (\xi_1 + g_1(\xi_2) + \psi'(g_1,g_2),g_1g_2)$$
for all $\xi_1,\xi_2 \in \omega(G) \otimes_{\mathbb Z} G$ and $g_1,g_2 \in G.$ Now, clearly
$[(0,g),(\xi,1)] = (g\xi - \xi,1)$, so that $\omega(G)^2 \otimes_{\mathbb Z} G$ lies in the commutator subgroup. Here $\omega(G)^2$ denotes the square of the ideal $\omega(G) \subseteq \mathbb Z[G]$. However, $\omega(G)/\omega(G)^2 = G$, so that the map to the abelianization factorizes via
$(G \otimes_{\mathbb Z} G) \rtimes_{\psi''} G$ with the two-cocycle $\psi''(g,h)= g \otimes h.$ This group is not yet abelian, but 
\begin{eqnarray*}
[(0,g),(0,h)] &=& (0,g)(0,h)(\psi''(g,g),-g)(\psi''(h,h),-h)\\
&=& (\psi''(g,h),g+h)(\psi''(g,g)+\psi(h,h)+ \psi''(g,h),-g-h) \\
&=& (\psi''(g,h) - \psi''(h,g),0).
\end{eqnarray*}
Hence, the abelianization factorizes via ${\rm Sym}^2(G) \rtimes_{\sigma} G$ with $\sigma \colon G \times G \to {\rm Sym}^2(G)$ given by $\sigma(g,h)=g \cdot h.$ This group is easily checked to be abelian and hence isomorphic to the abelianization of ${\rm Pol}_2(G)$. This finishes the proof.
\end{proof}

\section{Computations and consequences for non-commutative groups}\label{sec4}
\label{sec:cons}

This section contains some concrete computations of ${\rm Pol}_2(G)$ and its abelianization. The following corollary is already interesting for $d=1.$

\begin{corollary} \label{cor:free}
Let $\mathbb F_d$ be the free group on $d$ generators. The group ${\rm Pol}_2(\mathbb F_d)$ is isomorphic to
$(\mathbb Z[\mathbb F_d]^{\oplus d^2}) \rtimes \mathbb F_d$. Moreover, the abelianization is isomorphic to $\mathbb Z^{d^2+d}.$
\end{corollary}

\begin{proof} It is well-known that $\omega(\mathbb F_d) = \mathbb Z[\mathbb F_d]^{\oplus d}$ and hence $\omega(\mathbb F_d) \otimes_{\mathbb Z} \mathbb Z^d = \mathbb Z[\mathbb F_d]^{\oplus d^2}$ as $\mathbb F_d$-modules. Moreover, the cohomological dimension of $\mathbb F_d$ is less than two, i.e., all cohomology in dimension $2$ and above vanishes. Both of these facts are consequences of the existence of a free resolution
$$0 \to \mathbb Z[\mathbb F_d]^{\oplus d} \to \mathbb Z[\mathbb F_d] \stackrel{\varepsilon}{\to} \mathbb Z \to 0,$$
see, e.g., \cite{MR1324339}*{p.16}. Consequently, the class of the cocycle $$[\psi] \in {\rm H}^2(\mathbb F_d, \omega(\mathbb F_d) \otimes_{\mathbb Z} \mathbb F_d^{ab})$$ is trivial and the extension from Theorem \ref{thm:pol2} is split. 
This implies the claim.
\end{proof}

\begin{remark}
By Corollary \ref{cor:free}, ${\rm Pol}_2(\mathbb Z)$ is isomorphic to the lamplighter group $\mathbb Z \wr \mathbb Z$. It is well-known that this group is not nilpotent but residually nilpotent. As a consequence, for every $n \in \mathbb N$, there is a quadratic map $\varphi_n \colon \mathbb Z \to G_n$ with $G_n$ nilpotent of class $n$ and such that the image of $\varphi_n$ generates $G_n$.
\end{remark}

\begin{example}
The abelianization of the group ${\rm Pol}_2(\mathbb F_d)$ is $\mathbb Z^{d^2+d}$ by Corollary \ref{cor:free}, whereas the abelianization of ${\rm Pol}_2(\mathbb Z^d)$ is $\mathbb Z^{(d^2+3d)/2}$ by Corollary \ref{cor:abel}. The numbers in the exponent differ for $d \geq 2.$ This implies that there are exotic unital quadratic $\mathbb Z$-valued maps on $\mathbb F_2$ that do not factor through the abelianization $\mathbb F_2 \to \mathbb Z^2.$ One such map arises from the canonical homomorphism $\mathbb F_2$ to ${\rm H}_3(\mathbb Z) = \mathbb F_2/[\mathbb F_2,[\mathbb F_2,\mathbb F_2]]$ composed with the map ${\rm H}_3(\mathbb Z) \to \mathbb Z$ that sends the $3\times 3$-matrix to the entry in the upper right corner. More concretely, this map is given by the formula
$$\varphi(a^{n_1}b^{m_1} \cdots a^{n_k}b^{m_k}):= \sum_{i<j} m_in_j,$$
i.e., it counts the number of $b$'s on the left of $a$'s. To put this into perspective, Leibman \cite{MR1910931} showed that every unital quadratic map on any group with values in $\mathbb Z$ factors through a nilpotent quotient. 
\end{example}

\begin{proof}[Proof of Corollary \ref{cor:perfect}:]
By Theorem \ref{thm:pol2}, we have $$\ker(\pi \colon {\rm Pol}_2(G) \to G)= \omega(G) \otimes_{\mathbb Z} G^{ab}.$$ Hence, $\pi$ must be an isomorphism if $G$ is perfect. Thus, every unital quadratic map is a homomorphism. The general claim follows then by the inductive nature of the definition of the degree of a polynomial.
\end{proof}

Let's record another corollary of Theorem \ref{thm:pol2}.
\begin{corollary} \label{cor:finite}
If $G$ is finite, then ${\rm Pol}_2(G)$ is finite.
\end{corollary}
\begin{proof}
Indeed, $G^{ab}$ is finite, so that $\omega(G) \otimes_{\mathbb Z} G^{ab}$ is finite as well. It follows that ${\rm Pol}_2(G)$ is a finite group.
\end{proof}

\begin{remark}
While it is relatively easy to compute the abelianization of ${\rm Pol}_2(\mathbb Z/n \mathbb Z)$ from Corollary \ref{cor:abel}, it is more challenging to pin down ${\rm Pol}_2(\mathbb Z/n\mathbb Z)$. Clearly, Theorem \ref{thm:pol2} implies that ${\rm Pol}_2(\mathbb Z/n\mathbb Z)$ is a metabelian group of order $n^n$. We have ${\rm Pol}_2(\mathbb Z/2\mathbb Z) = \mathbb Z/4\mathbb Z$. One can compute ${\rm Pol}_2(\mathbb Z/3\mathbb Z)$ to be the Heisenberg group over the ring $\mathbb Z/3\mathbb Z.$
\end{remark}

\section{Classification of quadratic maps on abelian groups}\label{sec:abel}

In this section, we derive consequences of our main results for quadratic maps from an abelian group $G$ with target $H = S^1$. As an immediate consequence of Corollary~\ref{cor:abel}, Pontryagin duality, and Corollary~\ref{cor:finite}, we obtain the following description.

\begin{proposition}\label{prop:abel}
Let $G$ be an abelian group. The group ${\rm Quad}(G,S^1)$ is isomorphic to the Pontryagin dual of ${\rm Sym}^2(G)\rtimes_\sigma G$. In particular, if $G$ is finite, then there exists a non-canonical isomorphism of groups
\[
{\rm Quad}(G,S^1) \cong {\rm Sym}^2(G)\rtimes_\sigma G.
\]
\end{proposition}

We consider the preceding proposition as a computation of a quadratic Pontryagin dual of $G$. A natural question is to compute the higher-order Pontryagin dual beyond the quadratic case in concrete terms. This remains an interesting direction for future research.

\begin{example}\label{taylor}
    Let $d \geq 1$ be an integer. For $n = (n_1, \ldots, n_d) \in \mathbb{Z}^d$ and $j = (j_1, \ldots, j_d) \in \mathbb{Z}^d_{\geq 0}$, define the multi-binomial coefficient 
    \[
    B_j(n) \coloneqq \binom{n_1}{j_1} \cdots \binom{n_d}{j_d}. 
    \]
    Let $J$ denote the collection of $j = (j_1, \ldots, j_d) \in \mathbb{Z}^d_{\geq 0}$ such that $1 \leq \sum_{i=1}^d j_i \leq 2$. 
    We use the notation $\exp(x) \coloneqq e^{2\pi i x}$ to represent the fundamental character from $\mathbb{R}/\mathbb{Z}$ to $S^1$.  
    
    By the Taylor expansion formula for polynomial maps between filtered groups (see, e.g., \cite{gtz}*{Lemma B.9}, and its Appendix B for background on such polynomial maps), every quadratic $\varphi \in {\rm Quad}(\mathbb{Z}^d, S^1)$ is of the form 
    \[
    \varphi(n) = \prod_{j \in J} \exp(c_j)^{B_j(n)}
    \]    
    for some $c_j \in \mathbb{R}/\mathbb{Z}$, $j \in J$.  
    
    The cardinality of $J$, which is the number of basis polynomials $n \mapsto \exp(c_j)^{B_j(n)}$, is $(d^2 + 3d)/2$. 
    The abelianization of ${\rm Pol}_2(\mathbb{Z}^d)$ is $\mathbb{Z}^{(d^2+3d)/2}$, which is the Pontryagin dual of $(\mathbb{R}/\mathbb{Z})^{(d^2+3d)/2}$.
\end{example}

\begin{proposition}\label{prop:primary decomposition}
If $G$ is an abelian group of bounded exponent and $G=\oplus_p G_p$ is its primary decomposition, then
$${\rm Pol}_2^{ab}(G) = \oplus_p {\rm Pol}_2^{ab}(G_p).$$
\end{proposition}
\begin{proof}
This is a basic consequence of the isomorphism ${\rm Sym}^2(G) = \oplus_p {\rm Sym}^2(G_p)$ in this case.
\end{proof}

\begin{proposition}\label{prop-split} 
If $G$ is a $2$-divisible finite abelian group, then the extension in Corollary \ref{cor:abel} is split.
\end{proposition}
\begin{proof}
Note that if $G$ is finite abelian and divisible by $2$, then
$\nu \colon G \to {\rm Sym}^2(G)$ given by $\nu(g)=g^2/2$ satisfies
$\sigma(g,h)=\nu(g+h) - \nu(g) - \nu(h)$. In particular, the 2-cocycle $\sigma$ is a coboundary and the extension in Corollary \ref{cor:abel} is split.
\end{proof}

\begin{corollary} \label{cor-split}
If $G$ is a finite abelian and $2$-divisible group, then ${\rm Pol}_2^{ab}(G)= {\rm Sym}^2(G) \oplus G$. This applies to all finite abelian groups without 2-primary part.
\end{corollary}

In the remainder of this section, we prove Proposition \ref{classification}. 

Let $G$ be a $p$-primary group for some fixed prime $p$.  
Then $G$ is isomorphic to the direct sum $\bigoplus_{i=1}^d \mathbb{Z}/p^{n_i}\mathbb{Z}$ of cyclic groups of $p$-power order. By Proposition \ref{prop:abel}, we have the general description
\[
{\rm Quad}(G,S^1) \cong {\rm Sym}^2(G)\rtimes_\sigma G
\]
for the $2$-cocycle $\sigma\colon G\times G\to {\rm Sym}^2(G)$, given by $\sigma(g,h)=gh$. 

We can compute the symmetric second power of $G$ to be  
\[
{\rm Sym}^2(G) =  \bigoplus_{i=1}^d \mathbb{Z}/p^{n_i}\mathbb{Z} \oplus \bigoplus_{1\leq i<j\leq d} \mathbb{Z}/p^{\min(n_i,n_j)}\mathbb{Z},
\]
where we call the first summand the diagonal part and the second summand the off-diagonal part of ${\rm Sym}^2(G)$.
If $p\neq 2$, by Corollary \ref{cor-split}, we have directly:
\[
{\rm Sym}^2(G)\rtimes_\sigma G = \bigoplus_{i=1}^d \mathbb{Z}/p^{n_i}\mathbb{Z} \oplus \bigoplus_{1\leq i<j\leq d} \mathbb{Z}/p^{\min(n_i,n_j)}\mathbb{Z} \oplus \bigoplus_{i=1}^d \mathbb{Z}/p^{n_i}\mathbb{Z}
\] and this finishes the proof in this case.

\begin{remark}\label{tz}
Note that for odd primes and $G=(\mathbb{Z}/p\mathbb{Z})^d$, we obtain 
\[
{\rm Sym}^2(G)\rtimes_\sigma G = (\mathbb{Z}/p\mathbb{Z})^{\frac{d(d+3)}{2}}.
\]
This recovers, in the quadratic case, the description of Tao and Ziegler in \cite{tz2}*{Lemma 1.7}. Tao and Ziegler describe polynomial maps in terms of a closed-form formula for functions from $G$ to $S^1$. It is not difficult to read off from their functional description our above description in terms of a finite abelian group. 
\end{remark}

Now, we focus on $2$-primary groups $G$, in which case the computation is more subtle due to the existence of the twist. We start with the cyclic case. 

\begin{lemma}\label{lemB}
    Suppose $G=\mathbb{Z}/2^n\mathbb{Z}$. Then
    $
    {\rm Sym}^2(G)\rtimes_\sigma G 
    $
    is isomorphic to 
    $
    \mathbb{Z}/2^{n-1}\mathbb{Z} \oplus \mathbb{Z}/2^{n+1}\mathbb{Z}.
    $
\end{lemma}

\begin{proof}
    Note that ${\rm Sym}^2(G)\rtimes_\sigma G=\mathbb{Z}/2^n\mathbb{Z}\rtimes_\sigma \mathbb{Z}/2^n\mathbb{Z}$ is a finite abelian group with $2^{2n}$ elements. We prove the isomorphism by constructing two cyclic subgroups of $\mathbb{Z}/2^n\mathbb{Z}\rtimes_\sigma \mathbb{Z}/2^n\mathbb{Z}$ with orders $2^{n-1}$ and $2^{n+1}$ whose intersection is trivial. 
    
    Let $g$ denote the generator of $\mathbb{Z}/2^n\mathbb{Z}$. We claim that $$(0,g)^k=(\tbinom{k}{2} g^2, kg)$$ for every $k\geq 1$. Indeed, by induction, suppose that $(0,g)^{k-1}=(\binom{k-1}{2} g^2, (k-1)g)$ holds. Then 
    \[
    (0,g)^{k-1}(0,g)= (\tbinom{k-1}{2} g^2 + (k-1)g^2, (k-1)g+g) = (\tbinom{k}{2} g^2, kg),
    \]
    proving the claim. 
    Furthermore, we claim that $$(2g^2,2g)^k=(2k^2 g^2, 2k g)$$ for every $k\geq 1$. Again, by induction, suppose that $(2g^2,2g)^{k-1}=(2(k-1)^2 g^2,2(k-1)g)$ holds. Then 
    \[
         (2g^2,2g)^{k-1}(2g^2,2g)=(2(k-1)^2 g^2+ 2g^2+ 4(k-1)g^2,2kg) = (2k^2 g^2,(2k)g).
    \]
    Since the least $k$ such that $2^n$ divides $\binom{k}{2}$ is $2^{n+1}$, the order of $(0,g)$ is  $2^{n+1}$. Since the least $k$ such that $2^n$ divides $2k$ is $2^{n-1}$, the order of $(2g^2,2g)$ is $2^{n-1}$ . Since the only solution to $\binom{2k}{2}=2k^2$ is $k=0$, it follows that the cyclic subgroups generated by $(0,g)$ and $(2g^2,2g)$, respectively, have trivial intersection. 
\end{proof}

Having settled the cyclic case, we consider now the case $G= (\mathbb{Z}/2^{n}\mathbb{Z})^d$ with $n,d\geq 1$, Then, the group ${\rm Sym}^2(G)\rtimes_\sigma G$
has order $(2^n)^{\frac{d(d+3)}{2}}$.
If $gh$ is a basis element in the off-diagonal part of ${\rm Sym}^2(G)$, then $(gh,0)$ has order $2^n$ in ${\rm Sym}^2(G)\rtimes_\sigma G$. 
If $g$ is a basis element of $G$, by Lemma \ref{lemB}, $(0,g)$ has order $2^{n+1}$ in ${\rm Sym}^2(G)\rtimes_\sigma G$ and $(2g^2,2g)$ has order $2^{n-1}$ in ${\rm Sym}^2(G)\rtimes_\sigma G$. We can conclude, similarly to the proof of Lemma \ref{lemB}, on noting that cyclic groups of different basis elements intersect trivially,  that 
\[
  {\rm Sym}^2(G)\rtimes_\sigma G = (\mathbb{Z}/2^{n-1}\mathbb{Z})^d  \oplus (\mathbb{Z}/2^{n}\mathbb{Z})^{\binom{d}{2}}\oplus (\mathbb{Z}/2^{n+1}\mathbb{Z})^d.
 \]
This recovers, in the quadratic case, the classification of Tao and Ziegler in \cite{tz2}*{Lemma 1.7} for $G= (\mathbb{Z}/2\mathbb{Z})^d$.  

We can now prove the remaining part of Proposition \ref{classification}. Indeed, for general $G$, we have 
\[
{\rm Sym}^2(G)\rtimes_\sigma G = \left(\bigoplus_{i=1}^d \mathbb{Z}/2^{n_i}\mathbb{Z} \oplus \bigoplus_{1\leq i<j\leq d} \mathbb{Z}/2^{\min(n_i,n_j)}\mathbb{Z} \right)\rtimes_\sigma \bigoplus_{i=1}^d \mathbb{Z}/2^{n_i}\mathbb{Z}.
\]
Taking a basis element $gh$ in the off-diagonal part of ${\rm Sym}^2(G)$, say the generator of $\mathbb{Z}/2^{\min(n_i,n_j)}\mathbb{Z}$, then $(gh,0)$ has order $2^{\min(n_i,n_j)}$ in ${\rm Sym}^2(G)\rtimes_\sigma G$. By Lemma \ref{lemB}, if $g$ is a basis element of $G$, say the generator of $\mathbb{Z}/2^{n_i}\mathbb{Z}$, then $(0,g)$ has order $2^{n_i+1}$ in ${\rm Sym}^2(G)\rtimes_\sigma G$ and  $(2g^2,2g)$ has order $2^{n_i-1}$ in ${\rm Sym}^2(G)\rtimes_\sigma G$. Arguing similarly as in Lemma \ref{lemB} and the previous case, we conclude that 
\[
{\rm Sym}^2(G)\rtimes_\sigma G= \bigoplus_{i=1}^d \mathbb{Z}/2^{n_i-1}\mathbb{Z} \oplus \bigoplus_{1\leq i<j\leq d} \mathbb{Z}/2^{\min(n_i,n_j)}\mathbb{Z} \oplus \bigoplus_{i=1}^d \mathbb{Z}/2^{n_i+1}\mathbb{Z}.
\]

\section{Approximate polynomial maps}
\label{sec:apppol}

This is the first of two sections where we study the case of perfect groups more closely. We start with an investigation of approximate polynomial maps with values in groups that carry a bi-invariant metric. Let's recall the following definition:

\begin{definition}
We say that a perfect group $G$ is of commutator width $k$, if each element of $G$ is a product of at most $k$ commutators.
\end{definition}

\begin{remark}
\label{rem:simple}
Note that a finite group is automatically of finite commutator width. For interesting classes of finite perfect groups there exist uniform bounds on the commutator width. For example, as mentioned already in the introduction, Wilson \cite{wilson} showed that the commutator width is uniformly bounded for \emph{all} finite simple groups. In fact, the commutator width is now known to be exactly $1$, as established by the more recent solution to the Ore problem \cite{ore}.
See also \cite{nikolov} for even broader classes of finite perfect groups of bounded commutator width. It was shown by Holt--Plesken \cite{holtplesken}*{Lemma 2.1.10}, that the commutator width of a finite perfect group can become arbitrarily large. This also implies that an infinite perfect group does not need to have finite commutator width.
\end{remark}

We will now derive a useful consequence of Corollary \ref{cor:perfect}:

\begin{corollary} \label{cor:bound} Let $k \in \mathbb N.$
Let $G$ be a perfect group of commutator width $k$. Consider the free group on the set $\varphi(G)$. There exists a constant $c:=c(k)\geq 0$, such that for every $g,h \in G$, the element $\varphi(g_1g_2)\varphi(g_2)^{-1} \varphi(g_1)^{-1}$ is a product of $c$ conjugates of elements of the form
$$\varphi(g_1g_2g_3)^{-1} \varphi(g_1g_2)\varphi(g_2)^{-1}\varphi(g_1)^{-1} \varphi(g_1g_3)\varphi(g_3)^{-1}\varphi(g_2g_3)$$
or their inverses for $g_1,g_2,g_3 \in G$ .
\end{corollary}
\begin{proof} We consider subsets $\Sigma_1 := \{\varphi(g_1g_2)\varphi(g_2)^{-1} \varphi(g_1)^{-1} \mid g_1,g_2 \in G \}$ and $\Sigma_2 := \{\varphi(g_1g_2g_3)^{-1} \varphi(g_1g_2)\varphi(g_2)^{-1}\varphi(g_1)^{-1} \varphi(g_1g_3)\varphi(g_3)^{-1}\varphi(g_2g_3) \mid g_1,g_2,g_3 \in G\}$ of the free group generated by $\varphi(G)$.
Since $G$ is assumed to be perfect, by Corollary \ref{cor:perfect}, both sets of relations define the same normal subgroup, i.e., $\langle\! \langle \Sigma_1 \rangle\!\rangle = \langle\! \langle \Sigma_2 \rangle\!\rangle$. Thus, each word of the form $\varphi(g_1g_2)\varphi(g_2)^{-1} \varphi(g_1)^{-1}$ for $g_1,g_2 \in G$ must be a product of conjugates of elements of the form $\varphi(g_1g_2g_3)^{-1} \varphi(g_1g_2)\varphi(g_2)^{-1}\varphi(g_1)^{-1} \varphi(g_1g_3)\varphi(g_3)^{-1}\varphi(g_2g_3)$ or their inverses. The uniform estimate $c(k)$ on the number of factors follows from an easy ultra-product argument noting that the commutator width is preserved by the operation of taking an ultra-product of groups.
\end{proof}

Our focus is now on almost structure preserving maps to metric groups. Let $(H,\partial)$ be a metric group. It is of some interest in geometry as well as algebra to consider maps which are almost homomorphisms $\varphi \colon G \to H$ in the sense that
$$\partial(\varphi(gh),\varphi(g)\varphi(h))< \varepsilon$$ 
uniformly for all $g,h \in G$. This has been studied in various situations, see for example \cites{MR367862, MR693352, MR3733361, MR3867328, MR4634678}) and the references therein. Note that
$$(\Delta_{h}(\Delta_{g} \varphi))(1) = (\Delta_{g} \varphi)(h)(\Delta_{g} \varphi)(1)^{-1} = \varphi(gh) \varphi(h)^{-1} \varphi(g)^{-1}$$
for unital maps, so that being an almost homomorphism as above can be easily expressed in terms of a condition on iterated finite differences of $\varphi$. Theorem \ref{thm:approx} generalizes this idea and studies arbitrary iterated differences of $\varphi$. 

\begin{proof}[Proof of Theorem \ref{thm:approx}]
The proof proceeds by induction on $d$, where the case $d=1$ is trivial. The case $d=2$ requires special attention. Let's compute:
\begin{eqnarray*}
(\Delta_{g_1,g_2,g_3} \varphi)(g_0) &=& (\Delta_{g_1}(\Delta_{g_2,g_3}  \varphi))(g_0) \\
&=& (\Delta_{g_2,g_3}  \varphi)(g_1g_0) ((\Delta_{g_2,g_3}  \varphi)(g_0))^{-1} \\
&=&\varphi(g_3g_2g_1g_0) \varphi(g_2g_1g_0)^{-1} \varphi(g_1g_0) \varphi(g_3g_1g_0)^{-1} \\
&& \cdot (\varphi(g_3g_2g_0) \varphi(g_2g_0)^{-1} \varphi(g_0) \varphi(g_3g_0)^{-1})^{-1} \\
&=&\varphi(g_3g_2g_1g_0) \varphi(g_2g_1g_0)^{-1} \varphi(g_1g_0) \varphi(g_3g_1g_0)^{-1} \\
&& \cdot \varphi(g_3g_0)\varphi(g_0)^{-1}\varphi(g_2g_0)\varphi(g_3g_2g_0)^{-1}
\end{eqnarray*}
Thus, for $\varphi$ unital, we obtain:
\begin{equation*}
(\Delta_{g_1,g_2,g_3} \varphi)(1)=
\varphi(g_3g_2g_1) \varphi(g_2g_1)^{-1} \varphi(g_1) \varphi(g_3g_1)^{-1} \varphi(g_3)\varphi(g_2)\varphi(g_3g_2)^{-1}.
\end{equation*}
By our assumption, we obtain
$$\partial(\varphi(g_3g_2g_1) \varphi(g_2g_1)^{-1} \varphi(g_1) \varphi(g_3g_1)^{-1} \varphi(g_3)\varphi(g_2)\varphi(g_3g_2)^{-1},1) \leq \varepsilon$$ for all $g_1,g_2,g_3 \in G$. Note that the same estimate holds for conjugates and inverses by bi-invariance of the metric. Thus, Corollary \ref{cor:bound} implies that $\partial(\varphi(g_1g_2)\varphi(g_2)^{-1}\varphi(g_1)^{-1},1)\leq c(k) \varepsilon$ and equivalently
$$\partial((\Delta_{g_1,g_2}\varphi)(1),1) \leq c(k) \varepsilon.$$
As required, we conclude that $\varphi$ is a uniform $c(k)\varepsilon$-polynomial of degree $1$. The general case follows by induction. Indeed, note that for $d \geq 2$, $\varphi$ is a uniform $\varepsilon$-polynomial of degree $d+1$ if and only if $\Delta_g \varphi$ is a uniform $\varepsilon$-polynomial of degree $d$ for all $g \in G.$
\end{proof}

Note that the previous result is particularly useful when a stability result is present for uniform $\varepsilon$-homomorphisms, which is the case in various situations, see the references above. In those cases, every $\varepsilon$-polynomial for $\varepsilon>0$ small enough is uniformly close to a homomorphism.

A typical example where this happens is the group ${\rm U}(n)$ with the 2-norm and $G$ amenable, see \cite{MR3867328}. Here, for $a \in  M_n \mathbb C$ we denote by $\|a\|_2 = ((1/n)\sum_{ij} |a_{ij}|^2)^{\nicefrac12}$ the 2-norm or normalized Hilbert-Schmidt norm. Thus, we obtain:

\begin{theorem}
Let $d,k \geq 0$. For all $\varepsilon>0$, there exists $\delta:=\delta(d,k)>0$, such that the following holds: Let $G$ be an amenable group of commutator width $k$ and let $\varphi \colon G \to {\rm U}(n)$ be a unital uniform $\delta$-polynomial of degree $d$. Then there exists $n' \in [n,(1+\varepsilon)n] \cap \mathbb N$ and a homomorphism $\beta \colon G \to {\rm U}(n')$, such that
$\sup_g \|(\varphi(g) \oplus 1_{n'-n}) - \beta(g)\|_2 < \varepsilon.$
\end{theorem}

\begin{remark}
The previous corollary and theorem applies uniformly to all finite simple groups and fixed $d \in \mathbb N.$ An example of an infinite perfect amenable group with bounded commutator width is the group ${\rm Alt}_{\rm fin}(\mathbb N)$ consisting of all alternating permutations of $\mathbb N$ with finite support. In this case any homomorphism to ${\rm U}(n)$ is trivial by Jordan's theorem.
\end{remark}

\begin{remark}
We did not succeed to derive an explicit bound for $c(k)$, even in the case $k=1$. 
\end{remark}

\section{An inverse theorem for matrix-valued Gowers norms}
\label{sec:gowers}

We recall the definition of the Gowers norms of arbitrary degree for functions on a finite group with values in $M_n \mathbb C$.  
We set for a matrix valued function $\varphi \colon G \to M_n \mathbb C$
$$(\Delta_g \varphi)(h) := \varphi(gh)\varphi(h)^*.$$
Note that if the function takes values in the group of unitaries, then this is just the old finite difference.
We define 
$$\|\varphi\|^{2^k}_{U^k} :=\mathbb E_{g_0,g_1,\dots,g_{k}} {\rm tr}((\Delta_{g_1,\dots,g_k}\varphi)(g_{0})),$$
where $\Delta_{g_1,\dots,g_k} := \Delta_{g_1} \circ \cdots \circ \Delta_{g_k}$, ${\rm tr} \colon M_n \mathbb C \to \mathbb C$ denotes the normalized trace, and $\mathbb E$ denotes the expectation with respect to the normalized counting measure on $G^{k+1}$. 

For $k=2$, we compute
\begin{eqnarray*}
(\Delta_{g_1,g_2} \varphi)(g_0) &=& 
(\Delta_{g_2} \varphi)(g_1g_0) (\Delta_{g_2} \varphi)(g_0)^{*} \\
&=& \varphi(g_2g_1g_0) \varphi(g_1g_0)^{*} \varphi(g_0) \varphi(g_2g_0)^{*}
\end{eqnarray*}
so that
$${\mathbb E}_{g_0,g_1,g_{2}} {\rm tr}((\Delta_{g_1,g_2} \varphi)(g_0)) =  {\mathbb E}_{xy^{-1}zw^{-1}=1} {\rm tr}(\varphi(x)\varphi(y)^{*}\varphi(z)\varphi(w)^{*}).$$
We see from this equality that our definition agrees with the $U^2$-Gowers norm as defined in \cite{MR3733361}, modulo the normalization of the trace. The assignment $\varphi \mapsto \|\varphi\|_{U^k}$ is called \emph{$k$-th Gowers norm}, see Appendix \ref{app:gowers} for more details on its properties.

We start with the proof of Theorem \ref{thm:gowers}. First of all, let us note that we may assume without loss of generality that the function $\varphi$ takes unitary values and $\varphi(1)=1_n.$ Indeed, this is a consequence of the following two lemmas.

\begin{lemma} \label{lem:unitary}
Let $k \geq 1$ be a natural number. For all $\varepsilon>0$ there exists $\delta>0$ such that if
$v_1,\dots,v_k \in M_n \mathbb C$ are contractions with 
$${\rm tr}(v_1 \cdots v_k) \geq 1-\delta,$$
then there exists unitaries $u_1,\dots,u_k \in {\rm U}(n)$ such that $\|v_i - u_i\|_2 < \varepsilon$ for all $1 \leq i \leq k$ and
$u_1 \cdots u_k=1_n.$
\end{lemma}
\begin{proof}
Suppose by contradiction that for some $\varepsilon>0$, there exists sequences of contractions
$(v_{i,m})_m$ for $1 \leq i \leq k$ with $v_{i,n} \in M_{n_m} \mathbb C$ such that
$${\rm tr}(v_{1,m} \cdots v_{k,m}) \geq 1 - \frac1m$$ and at the same time, there do not exist unitaries $u_1,\dots,u_k$ which are $\varepsilon$-close to $v_1,\dots,v_k$ in the 2-norm. 

Consider the tracial ultraproduct of von Neumann algebras.
$$(M,\tau) := \prod_{m \to \mathcal U} (M_{n_m} \mathbb C,{\rm tr}).$$
It is well-known that $(M,\tau)$ is a tracial von Neumann algebra, in fact it is a factor of type II$_1$. Consider the contractions $v_i \in M$ represented by the sequences $(v_{i,m})_m$. From our assumption, we obtain
$$\tau(v_1 \cdots v_k)=1.$$
Now, let $x \in M$ be an arbitrary contraction with $\tau(x)=1.$
Since,
$\|1-x\|^2_2 = 1 + \tau(x^*x) - \tau(x) - \tau(x^*) \leq 0$
we obtain $x=1.$ In particular, $v_1$ is a contraction whose inverse is also a contraction, being equal to $v_2\cdots v_k$. This implies that $v_1$ is isometric and hence unitary. By induction, we conclude that $v_1,\dots,v_k$ are all unitaries. We conclude that $v_i$ can also be represented by a sequence of unitaries $(u_{i,m})_m$ with $$\lim_{m \to \mathcal U} \|u_{i,m} - v_{i,m}\|_2 = 0 \quad \mbox{for} \quad 1 \leq i \leq k-1.$$ We put $u_{i,k}:=(u_{i,1}\cdots u_{i,k-1})^{-1}$. This is a contradiction and thus, proves the claim.
\end{proof}

\begin{lemma} Let $k \in \mathbb N$ and $\varepsilon>0$. There exists $\delta:=\delta(k,\varepsilon)>0$, such that the following is true: Let $G$ be a finite group and let $\varphi \colon G \to M_n \mathbb C$ with
$$\|\varphi\|_{\infty} \leq 1 \quad \mbox{and} \quad \|\varphi\|^{2^k}_{U^k} > 1 -\delta.$$
There exists a function $\varphi' \colon G \to {\rm U}(n)$ such that
$$\mathbb E_g \|\varphi(g) - \varphi'(g)\|^2_2 < \varepsilon \quad \mbox{and} \quad \|\varphi'\|^{2^k}_{U^k} > 1 -\varepsilon.$$
\end{lemma}
\begin{proof} We have 
$$\|\varphi\|^{2^{k}}_{U^k} = \mathbb E_{g_0,g_2,g_3,\dots,g_{k}} {\rm tr}((\Delta_{g_1,\dots,g_k}\varphi)(g_{0})) > 1 -\delta$$ while $(\Delta_{g_1,\dots,g_k}\varphi)(g_{0})$ is a product of $2^k$ contractions from the set $\varphi(G)$. The lower bound implies that the set $\Sigma \subset G$ of those $g \in G$ for which $\varphi(g)$ appears in such a product whose trace is at least $1 -\delta^{\nicefrac12}$ is of density at least $1 -\delta^{\nicefrac12}$. By Lemma \ref{lem:unitary}, provided $\delta>0$ is small enough, any such $\varphi(g)$ can be replaced by a unitary $\varphi'(g)$ such that $\|\varphi(g)-\varphi(g')\|_2< \delta'.$ For all $g \not \in \Sigma,$ we define $\varphi'(g)=1_n.$

Now, note that the set of $(k+1)$-tuples $g_0,\dots,g_k$ such that all contractions arising in factors of $(\Delta_{g_1,\dots,g_k}\varphi)(g_{0})$ lie in $\Sigma$ is of density at least $1-2^k\delta^{\nicefrac12}$. For those $(k+1)$-tuples, we obtain
$$\|(\Delta_{g_1,\dots,g_k}\varphi)(g_{0}) - (\Delta_{g_1,\dots,g_k}\varphi')(g_{0})\|_2 \leq 2^k \delta'$$
and hence
$$|{\rm tr}((\Delta_{g_1,\dots,g_k}\varphi)(g_{0}) - (\Delta_{g_1,\dots,g_k}\varphi')(g_{0}))| \leq 2^{k+1} \delta'.$$
Finally, we conclude
$$\left| \|\varphi\|_{U^k}^{2^k} - \|\varphi'\|^{2^k}_{U^k} \right| < 2^{k+1} \delta' + 2^{k+1}\delta^{\nicefrac12}.$$
Now, $\delta'>0$ is as small as we wish, when $\delta$ is small. And $\delta>0$ can be taken small enough so that also the second summand does not contribute. This proves the claim.
\end{proof}

Before we start with the actual proof of Theorem \ref{thm:gowers}, let us recall that the case $k=2$ is a consequence of \cite{MR3733361}*{Theorem 5.6}. Indeed, assuming that $\varphi$ takes values in ${\rm U}(n)$, we compute
$$\|(\varphi(g) \oplus 1_{n'-n}) - \beta(g)\|^2_2 = 2 - 2 \Re \langle \varphi(g) \oplus 1_{n'-n}, \beta(g) \rangle,$$
so that large correlation of unitary-valued maps is directly tied to a small 2-norm of the difference.

\begin{proof}[Proof of Theorem \ref{thm:gowers}] We may assume that $\varphi$ takes values in unitaries and upon replacing $\varphi$ by $g \mapsto \varphi(1)^{*} \varphi(g)$, we may assume that $\varphi$ is unital.

We prove the claim by induction on $k$, where the case $k=2$ was proved in \cite{MR3733361} as noted above. So let $k \geq 3$, $\delta>0$ and $\varphi \colon G \to {\rm U}(n)$ be arbitrary and assume that $\|\varphi\|^{2^k}_{U^k}  \geq 1 - \varepsilon$ for some $\varepsilon>0$ to be determined later.  Now, the inequality just says that 
$$\mathbb E_{g_0,g_2,g_3,\dots,g_{k}} {\rm tr}((\Delta_{g_1,\dots,g_k}\varphi)(g_{0})) = \|\varphi\|^{2^k}_{U^k}  \geq 1 - \varepsilon.$$

Note that $\|\varphi\|_{\infty} \leq 1$ and we conclude by Markov's inequality, that the set $A \subseteq G$ of those $g \in G$, such that
$$\mathbb E_{g_0,g_2,g_3,\dots,g_{k-1}} {\rm tr}((\Delta_{g_1,\dots,g_{k-1},g}\varphi)(g_{0})) = \|\Delta_{g}\varphi\|^{2^{k-1}}_{U^{k-1}}  \geq 1 - \varepsilon^{\nicefrac12}$$
has cardinality at least $(1- \varepsilon^{\nicefrac12})|G|.$  

Let $\eta>0$ to be determined later. For $\varepsilon>0$ such that $\varepsilon^{\nicefrac12} \leq \varepsilon(\eta,k-1,k_0)$, by induction, we find $n' \in [n,(1+ \eta)n] \cap \mathbb N$, homomorphisms $\beta_g \colon G \to {\rm U}(n')$ such that
$$\mathbb E_h \|(\varphi(g)^*(\Delta_g\varphi)(h) \oplus 1_{n'-n}) - \beta_g(h)\|^2_2 < \eta$$
for all $g \in A.$ 

For $g \in A$, again by Markov's inequality, there exists a subset $B(g) \subseteq G$ of density at least $1-\eta^{\nicefrac12}$ such that $\|(\varphi(g)^*(\Delta_g\varphi)(h) \oplus 1_{n'-n}) - \beta_g(h)\|_2< \eta^{\nicefrac14}$ for all $h \in B(g).$

Let us consider $g \in A \cap A^{-1}$. For ease of notation, we will write $u\stackrel{\nu}{=} v$ if $\|u-v\|_2< \nu.$
Now, by the proof of Equation \eqref{eq:cocycbeta}, we obtain
$$1_{n'} \stackrel{2\eta^{\nicefrac14}}{=} \varphi(g)^{-1} \beta_{g^{-1}}(h) \varphi(g) \cdot \beta_g(h)$$
for all $h \in g^{-1}B(g^{-1}) \cap B(g)$.
Thus, if $h_1,h_2,h_2h_1 \in g^{-1}B(g^{-1}) \cap B(g)$, we obtain
\begin{eqnarray*}
\beta_g(h_1h_2) &=& \beta_g(h_1)\beta_g(h_2) \\
&\stackrel{4\eta^{\nicefrac14}}{=}& (\varphi(g)^{-1} \beta_{g^{-1}}(h_1) \varphi(g))^{-1}(\varphi(g)^{-1} \beta_{g^{-1}}(h_2) \varphi(g))^{-1} \\
&=& (\varphi(g)^{-1} \beta_{g^{-1}}(h_2)\beta_{g^{-1}}(h_1) \varphi(g))^{-1}\\
&=& (\varphi(g)^{-1} \beta_{g^{-1}}(h_2h_1) \varphi(g))^{-1} \\
&\stackrel{2\eta^{\nicefrac14}}{=}& \beta_g(h_2h_1).
\end{eqnarray*}

We conclude from this computation that the set $C(g)$ of pairs $(h_1,h_2) \in G\times G$ with $$\beta_g([h_1,h_2]) \stackrel{6\eta^{\nicefrac14}}{=} 1_{n'}$$ has density at least $1-6\eta^{\nicefrac12}.$ If $1-6\eta^{\nicefrac12}>\nicefrac12$, then $C(g)^2=G \times G.$ Indeed, we have $C(g)^{-1}= C(g)$ and under this condition on the intersection $(h_1,h_2)C(g)^{-1} \cap C(g)$ must be non-empty for each $(h_1,h_2) \in G \times G$. Note also that commutators satisfy the following universal identities:
$$[ab,cd]=[ab,c]\cdot [ab,d]^{c} = [a,c]^b[b,c][a,d]^{bc}[b,d]^c$$
for all $a,b,c,d \in G$. Thus, we obtain from $C(g)^2 = G \times G$ that 
$\beta_g([h_1,h_2]) \stackrel{24\eta^{\nicefrac14}}{=} 1_{n'}$
for all $h_1,h_2 \in G.$

Since $G$ has commutator width bounded by $k_0$, this implies $\|\beta_g(h) - 1_{n'}\| < 24k_0 \eta^{\nicefrac14}$ for all $h \in G.$
Thus, $\Delta_g \varphi$ is almost constant for $g \in A \cap A^{-1}$. Since $A \cap A^{-1}$ has density at least $1- 2\varepsilon^{\nicefrac12}$ and $\|\varphi\|_{\infty}\leq 1$, this implies that $\|\varphi\|_{U^2}$ is close to $1$. Indeed, we obtain
$$\|\varphi\|^4_{U^2} = \mathbb E_{g_0,g_1,g_2} {\rm tr}( \Delta_{g_1,g_2} \varphi)(g_0)) 
\geq 1 - 48k_0 \eta^{\nicefrac14} - 4\varepsilon^{\nicefrac12}.$$
Thus, taking $\eta$ and $\varepsilon$ sufficiently small, so that $48k_0 \eta^{\nicefrac14} + 4\varepsilon^{\nicefrac12} \leq \varepsilon(\delta,2,k_0)$  we are done by the results of Gowers--Hatami for the case $k=2$.
\end{proof}

\begin{remark}
The result of Gowers--Hatami \cite{MR3733361}*{Theorem 5.6} for $k=2$ is more general since it is also an inverse theorem in the 1\% regime, i.e., a positive correlation with a linear map is derived only from $\|\varphi\|_{U^2} \geq \varepsilon.$ We were not able to generalize this to hold for arbitrary $k \geq 3$.
\end{remark}

\appendix
\section{The Gowers norms on matrix valued functions}
\label{app:gowers}
In this appendix, extending results from \cite{MR3733361} for the case $k=2$, we study basic properties for the Gowers norms on matrix-valued functions on a finite group.  Motivated by the combinatorial structure of iterated finite differences, we define non-abelian cubes as follows:
\begin{eqnarray*}
c_0&=&(1),\\
c_1(g_1)&=&(g_1,1),\\
c_2(g_1,g_2)&=&(g_1g_2,g_2,1,g_1),\\
c_3(g_1,g_2,g_3)&=&(g_1g_2g_3,g_2g_3,g_3,g_1g_3,g_1,1,g_2,g_1g_2),
\end{eqnarray*}
where each $c_k$ is a list of $2^{k}$ words in the variables $g_1,\dots,g_k$.
In general, we set:
$$c_{k+1}(g_1,\dots,g_{k+1}) = (c_k(g_1,\dots,g_k)g_{k+1},\bar c_k(g_1,\dots,g_k)),$$
where $\bar c$ denotes the reversed list. We write $$c_k(g_1,\dots,g_k)=(c_{1,k}(g_1,\dots,g_k),\dots,c_{2^k,k}(g_1,\dots,g_k)).$$
\begin{definition}
Let $G$ be a finite group $n,k \in \mathbb N$. Let $f_i \colon G \to M_n \mathbb C$ for $1 \leq i \leq 2^k$. The Gowers inner product is defined as:
$$\left((f_i)_{i=1}^{2^k} \right):= \mathbb E_{g_0,\dots, g_k} {\rm tr} \left(\prod_{i=1}^{2^k} f_i\left((c_{i,k}(g_1,\dots, g_k) g_0)\right)^{*^{i-1}} \right),$$
where $\prod$ denotes the left-ordered product, ${\rm tr}$ denotes the normalized trace, and $x^*$ is the adjoint of $x \in M_n \mathbb C$.
\end{definition} 

Let $\varphi \colon \{1,\dots,2^k\} \to \{1,\dots,2^k\}$ be a function. We define the functions
$$L(i) =\begin{cases} i & 1 \leq i \leq 2^{k-1} \\ 2^k+1-i & 2^{k-1}+1 \leq i \leq 2^k.\end{cases}$$ and $$R(i) =\begin{cases} 2^k+1-i & 1 \leq i \leq 2^{k-1} \\ i & 2^{k-1}+1 \leq i \leq 2^k.\end{cases}$$
and set $\varphi':= \varphi \circ L$ and $\varphi'' := \varphi \circ R.$
Moreover, we define 
$$(S\varphi)(i) = \begin{cases}
\varphi(2^k) & i=1 \\
\varphi(i-1) & 2 \leq i \leq 2^{k} \end{cases}$$
and $\bar \varphi(i) := 2^k +1 -i.$

It is a basic fact that $((f_{\varphi(i)})_{i=1}^{2^k}) \geq 0$ if $\varphi = \bar \varphi$. The following lemma relates the Gowers inner product to the Gowers norms.

\begin{lemma} Let $G$ be a finite group, $n,k \in \mathbb N$. Let $f \colon G \to M_n \mathbb C$. Then, we have
$$\|f\|_{U^k} = \left((f)_{i=1}^{2^k}\right)^{\nicefrac1{2^k}}.$$
\end{lemma}

\begin{lemma} \label{lem:gowfaith}
Let $G$ be a finite group, $n,k \in \mathbb N$. Let $f \colon G \to M_n \mathbb C$. If $k \geq 2$, then
$\|f\|_{U^k}=0$ implies $f=0.$
\end{lemma}
\begin{proof}
We prove the claim by induction starting with $k=2$. Note that 
\begin{eqnarray*}
\|f\|^4_{U^2} &=& \mathbb E_{g_0,g_1,g_2} {\rm tr}(f(g_1g_2g_0)f(g_2g_0)^* f(g_0)f(g_1g_0)^* )\\
&=& \mathbb E_{g_0,g_1,g_2} {\rm tr}(f(g_1g_2)f(g_2)^* f(g_0)f(g_1g_0)^* )\\
&=& \mathbb E_{g_1}   \|\mathbb E_{g} f(g_1g)f(g)^*\|^2_{2} 
\end{eqnarray*}
In particular, $\|f\|_{U^2}=0$ implies $\mathbb E_{g} f(g_1g)f(g)^*=0$ for all $g_1 \in G$. Thus, considering $g_1=1$, we conclude $f$ is zero. Hence, $\|.\|_{U^2}$ is faithful on the space of matrix-valued functions on $G$. 

Let's now assume that the claim is already established for $k-1$ and consider the $k$-th Gowers norm. It follows from the definition of the Gowers norms that
$\|f\|^{2^k}_{U^k} = \mathbb E_g \|\Delta_g f\|^{2^{k-1}}_{U^{k-1}}.$ In particular, if $\|f\|_{U^k}=0$, then $\Delta_g f =0$ for all $g \in G$ by induction. Since $(\Delta_1 f)(h) = f(h)f(h)^*$, this implies that $f$ is zero.
\end{proof}

\begin{lemma}
Let $G$ be a finite group, $n,k \in \mathbb N$. Let $f_i \colon G \to M_n \mathbb C$ for $1 \leq i \leq 2^k$ and let $\varphi \colon \{1,\dots,n\} \to \{1,\dots,n\}$ be some function. Then, we have

\begin{eqnarray}
|((f_{\varphi(i)})_{i=1}^{2^k})|&\leq & ((f_{\varphi'(i)})_{i=1}^{2^k})^{1/2} \label{claim1}((f_{\varphi''(i)})_{i=1}^{2^k})^{1/2} \\
((f_{(S\varphi)(i)})_{i=1}^{2^k}) &=& \overline{((f_{\varphi(i)})_{i=1}^{2^k})} \label{claim2}
\end{eqnarray}
\end{lemma}
\begin{proof}
First of all, note that it is enough to prove the inequality in \eqref{claim1} for $\varphi$ being the identity function.

In the following computation, we use the shorthand $\vec{g} = (g_1,\dots,g_{k-1})$ and write for $m:=2^{k-1}$ and functions $f_i \colon G \to M_k \mathbb C$ for $1 \leq i \leq m$:
$$f_{i_1,\dots,i_{m}}(h_1,\dots,h_{m}) := \prod_{i=1}^{m} f_{i_1}(h_i)^{*^{i}}.$$ Note that in this notation, we have
$$(f_{i_1,\dots,i_{m}}(h_1,\dots,h_m))^* = f_{i_{m},\dots,i_1}(h_{m},\dots,h_1).$$
Claim (1) is a consequence of the usual Cauchy-Schwarz inequality carefully applied:
\begin{eqnarray*}  |(f_{i})_{i=1}^{2^k})|  &=& \mathbb |\mathbb E_{\vec{g},g_0,g_k} {\rm tr}(f_{1,\dots,2^{k-1}}(c_{k-1}(\vec{g}g_k g_0) f_{2^{k-1}+1,\dots,2^{k}}(\bar c_{k-1}(\vec{g}g_0))|\\
&=& |\mathbb E_{\vec{g},g_0,g_k} {\rm tr}(f_{1,\dots,2^{k-1}}(c_{k-1}(\vec{g}g_k) f_{2^{k-1}+1,\dots,2^{k}}(\bar c_{k-1}(\vec{g}g_0))|\\
&=& |\mathbb E_{\vec{g}} {\rm tr}\left( (\mathbb E_{g_k} f_{1,\dots,2^{k-1}}(c_{k-1}(\vec{g}g_k)) \cdot (\mathbb E_{g_0} f_{2^{k},\dots,2^{k-1}+1,}( c_{k-1}(\vec{g}g_0))^* \right) |\\
&\leq& \left ( \mathbb E_{\vec{g}} {\rm tr}(|\mathbb E_{g_k}f_{1,\dots,2^{k-1}}(c_{k-1}(\vec{g}g_k)|^2) \right)^{1/2} \cdot \\
&& \left( \mathbb E_{\vec{g}} {\rm tr}(|\mathbb E_{g_0} f_{2^{k},\dots,2^{k-1}+1,}(c_{k-1}(\vec{g}g_0)|^2) \right)^{1/2} \\
&=& \left ( \mathbb E_{\vec{g}} {\rm tr}(\mathbb E_{g_k,g'_0}f_{1,\dots,2^{k-1}}(c_{k-1}(\vec{g}g_k)f_{1,\dots,2^{k-1}}(c_{k-1}(\vec{g}g'_0)^*) \right)^{1/2} \cdot \\
&& \left( \mathbb E_{\vec{g}} {\rm tr}(\mathbb E_{g_0,g'_k} f_{2^{k},\dots,2^{k-1}+1,}(c_{k-1}(\vec{g}g_0)f_{2^{k},\dots,2^{k-1}+1,}( c_{k-1}(\vec{g}g'_k)^*) \right)^{1/2} \\
&=& \left ( \mathbb E_{\vec{g},g_k,g'_0} {\rm tr}(f_{1,\dots,2^{k-1}}(c_{k-1}(\vec{g}g_k)f_{2^{k-1},\dots,1}(\bar c_{k-1}(\vec{g}g'_0)) \right)^{1/2} \cdot \\
&& \left( \mathbb E_{\vec{g},g_0,g'_k} {\rm tr}( f_{2^{k},\dots,2^{k-1}+1,}( c_{k-1}(\vec{g}g_0)f_{2^{k-1}+1,\dots,2^{k}}( \bar c_{k-1}(\vec{g})g'_k)) \right)^{1/2} \\
&=& ((f_{L(i)})_{i=1}^{2^k})^{1/2}((f_{R(i)})_{i=1}^{2^k})^{1/2}.
\end{eqnarray*}
Claim \eqref{claim2} is an immediate consequence of the trace property, i.e., ${\rm tr}(xy)={\rm tr}(yx)$, and the fact that the trace of the hermitean conjugate of a matrix is just the complex conjugate of the trace of the matrix.
\end{proof}

\begin{proposition}[Gowers-Cauchy-Schwarz inequality] Let $G$ be a finite group, $n,k \in \mathbb N$. Let $f_i \colon G \to M_n \mathbb C$ for $1 \leq i \leq 2^k$. Then
$$|((f_i)_{i=1}^{2^k})| \leq \prod_{i=1}^{2^k} \|f_i\|_{U^k}.$$
\end{proposition}
\begin{proof}
Let's first assume that $\|f_i\|_{U^k} = 1$ for all $1 \leq i \leq 2^k.$ We prove the stronger claim
$|((f_{\varphi(i)})_{i=1}^{2^k})| \leq 1$ for all functions $\varphi \colon \{1,\dots,2^k\} \to \{1,\dots 2^k\}.$ 
Clearly, the claim is true for all constant functions by definition of the Gowers norm.

Set $\alpha(\varphi):=|((f_{\varphi(i)})_{i=1}^{2^k})|$ and $\alpha := \max_{\varphi} \alpha(\varphi).$ Let $\Phi$ be the set of functions $\varphi$ that realize the maximum. If $\varphi$ is a maximizer, it follows from the previous lemma that
$\alpha(\varphi) \leq \alpha(\varphi')^{1/2} \alpha(\varphi'')^{1/2} \leq \alpha(\varphi')^{1/2} \alpha(\varphi)^{1/2}$. Thus, $\alpha(\varphi) = \alpha(\varphi')$ and $\varphi'$ is another maximizer; hence $\Phi$ is closed under the operation $\varphi \mapsto \varphi'$. It is also clear that the set of maximizers is closed under $S$. Now, the chain
$\varphi \mapsto S\varphi' \mapsto S^2(S \varphi')' \mapsto S^4(S^2(S \varphi')')' \mapsto \cdots$
has the effect that after $i$ steps, the values of the first $2^{i}$ entries are equal to $\varphi(1).$ Thus, after $k$ steps, the resulting function is $c_{\varphi(1)}$, i.e, the constant function with value equal to $\varphi(1).$ This proves that
$\alpha = \alpha(\varphi) = \alpha(c_{\varphi(1)})=1$ and we are done. After scaling, this also covers the case when $\|f_i\|_{U^k} \neq 0$ for all $1 \leq i \leq 2^k.$ If $\|f_i\|_{U^k}=0$ for one of the $f_i's$, then we are done by Lemma \ref{lem:gowfaith}.
\end{proof}

\begin{corollary} Let $G$ be a finite group, $n,k \in \mathbb N$. Let $f \colon G \to M_n \mathbb C$. We have:
$\|f\|_{U^{k+1}} \geq \|f\|_{U^k}$
\end{corollary}
\begin{proof} This follows from the computation
$$\|f\|^{2^k}_{U^k} =  (\underbrace{f,\dots,f}_{2^k},\underbrace{1,\dots,1}_{2^k})  \leq \|f\|^{2^k}_{U^{k+1}} \cdot \|1\|^{2^k}_{U^{k+1}} = \|f\|^{2^k}_{U^{k+1}},$$
where we used the Gowers-Cauchy-Schwarz inequality for the $(k+1)$th Gowers inner product.
\end{proof}

Finally, we are able to prove the triangle inequality.
\begin{corollary}
Let $G$ be a finite group, $n,k \in \mathbb N$. Let $f_1,f_2 \colon G \to M_n \mathbb C$. We have
$$\|f_1 + f_2\|_{U^k} \leq \|f_1\|_{U^k} + \| f_2\|_{U^k}.$$
In particular, for $k \geq 2$, $f \mapsto \|f\|_{U^k}$ defines a norm on the space of matrix-valued functions on $G$.
\end{corollary}
\begin{proof}
This part of the argument is unchanged in comparison to the more classical instances of the triangle inequality for the Gowers norm. Indeed, we have
\begin{eqnarray*}
\|f_1+f_2\|^{2^k}_{U^k} &=& ((f_1+f_2)_{i=1}^{2^k}) \\
&=& \sum_{\varphi \colon \{1,\dots,2^k\} \to \{1,2\}} ((f_{\varphi(i)})_{i=1}^{2^k}) \\
& \leq& \sum_{\varphi \colon \{1,\dots,2^k\} \to \{1,2\}} 
 \|f_1\|^{|\varphi^{-1}(1)|}_{U^k} \|f_2\|^{|\varphi^{-1}(2)|}_{U^k} \\
&=& \sum_{m=0}^{2^k} \binom{2^k}{m} \|f_1\|_{U^k}^m \|f_2\|_{U^k}^{2^k-m} \\
&=& (\|f_1\|_{U^k} + \|f_2\|_{U^k})^{2^k} 
\end{eqnarray*}
In order to obtain that $f \mapsto \|f\|_{U^k}$ is indeed a norm, we only need Lemma \ref{lem:gowfaith} and observe that $\|\lambda f\|_{U^k} = |\lambda| \|f\|_{U^k}$ holds, which is obvious.
\end{proof}

\section*{Acknowledgments}
Asgar Jamneshan acknowledges funding by the Deutsche Forschungsgemeinschaft (DFG, German Research Foundation) - 547294463. 
Andreas Thom acknowledges funding by the Deutsche Forschungsgemeinschaft (SPP 2026 ``Geometry at infinity''). We are grateful to the anonymous referee for constructive feedback and many helpful remarks, which helped improve the exposition of the paper.


\begin{bibdiv}
\begin{biblist}

\bib{MR4634678}{article}{
  author={Becker, Oren},
  author={Chapman, Michael},
  title={Stability of approximate group actions: uniform and probabilistic},
  journal={J. Eur. Math. Soc. (JEMS)},
  volume={25},
  date={2023},
  number={9},
  pages={3599--3632},
}

\bib{MR1324339}{book}{
  author={Brown, Kenneth S.},
  title={Cohomology of groups},
  series={Graduate Texts in Mathematics},
  volume={87},
  note={Corrected reprint of the 1982 original},
  publisher={Springer-Verlag, New York},
  date={1994},
  pages={x+306},
}

\bib{MR3867328}{article}{
  author={De Chiffre, Marcus},
  author={Ozawa, Narutaka},
  author={Thom, Andreas},
  title={Operator algebraic approach to inverse and stability theorems for amenable groups},
  journal={Mathematika},
  volume={65},
  date={2019},
  number={1},
  pages={98--118},
}

\bib{g1}{article}{
  author={Gowers, Timothy},
  title={A new proof of Szemerédi's theorem},
  journal={Geom. Funct. Anal.},
  volume={11},
  number={3},
  pages={465--588 (2001); erratum 11, no. 4, 869},
  year={2001},
}

\bib{MR3733361}{article}{
  author={Gowers, Timothy},
  author={Hatami, Omid},
  title={Inverse and stability theorems for approximate representations of finite groups},
  language={Russian, with Russian summary},
  journal={Mat. Sb.},
  volume={208},
  date={2017},
  number={12},
  pages={70--106},
  translation={
    journal={Sb. Math.},
    volume={208},
    date={2017},
    number={12},
    pages={1784--1817},
  },
}

\bib{gtz}{article}{
  author={Green, Ben},
  author={Tao, Terence},
  author={Ziegler, Tamar},
  title={An inverse theorem for the Gowers $U^{k+1}[N]$-norm},
  journal={Ann. Math.},
  volume={176},
  year={2012},
  pages={1231--1372},
}

\bib{MR367862}{article}{
  author={Grove, Karsten},
  author={Karcher, Hermann},
  author={Ruh, Ernst A.},
  title={Group actions and curvature},
  journal={Bull. Amer. Math. Soc.},
  volume={81},
  date={1975},
  pages={89--92},
}

\bib{holtplesken}{book}{
  author={Holt, Derek},
  author={Plesken, Wilhelm},
  title={Perfect groups},
  publisher={Oxford Math. Monogr., Oxford Univ. Press, New York},
  date={1989},
}

\bib{jst}{article}{
  author={Jamneshan, Asgar},
  author={Shalom, Or},
  author={Tao, Terence},
 title = {The structure of totally disconnected {Host}--{Kra}--{Ziegler} factors, and the inverse theorem for the ${U}^k$ {Gowers} uniformity norms on finite abelian groups of bounded torsion},
 year = {2023},
 note = {Preprint, arXiv:2303.04860},
 url = {https://arxiv.org/abs/2303.04860},
}

\bib{jt}{article}{
  author={Jamneshan, Asgar},
  author={Tao, Terence},
  title={The inverse theorem for the $U^3$ Gowers uniformity norm on arbitrary finite abelian groups: Fourier-analytic and ergodic approaches},
  journal={Discrete Anal.},
  year={2023},
  pages={Paper No. 11, 48},
}

\bib{MR693352}{article}{
  author={Kazhdan, David},
  title={On $\varepsilon$-representations},
  journal={Israel J. Math.},
  volume={43},
  date={1982},
  number={4},
  pages={315--323},
}

\bib{MR1608723}{article}{
  author={Leibman, Alexander},
  title={Polynomial sequences in groups},
  journal={J. Algebra},
  volume={201},
  date={1998},
  number={1},
  pages={189--206},
}

\bib{MR1910931}{article}{
  author={Leibman, Alexander},
  title={Polynomial mappings of groups},
  journal={Israel J. Math.},
  volume={129},
  date={2002},
  pages={29--60},
}

\bib{leng}{article}{
 author = {Leng, James},
 author = {Sah, Ashwin},
 author = {Sawhney, Mehtaab},
 title = {Quasipolynomial bounds on the inverse theorem for the {Gowers} ${U}^{s+1}[{N}]$-norm},
 year = {2024},
 note = {Preprint, arXiv:2402.17994},
 url = {https://arxiv.org/abs/2402.17994},
}

\bib{ore}{article}{
  author={Liebeck, Martin W.},
  author={O'Brien, E. A.},
  author={Shalev, Aner},
  author={Tiep, Pham Huu},
  title={The Ore conjecture},
  journal={J. Eur. Math. Soc. (JEMS)},
  volume={12},
  date={2010},
  number={4},
  pages={939--1008},
}

\bib{luka}{article}{
 author = {Mili{\'c}evi{\'c}, Luka},
 title = {Quasipolynomial inverse theorem for the $U^4(\mathbb{F}_p^n)$ norm},
 year = {2024},
 note = {Preprint, arXiv:2410.08966},
 url = {https://arxiv.org/abs/2410.08966},
}

\bib{nikolov}{article}{
  author={Nikolov, Nikolay},
  title={On the commutator width of perfect groups},
  journal={Bull. London Math. Soc.},
  volume={36},
  year={2004},
  pages={30--36},
}

\bib{tz1}{article}{
  author={Tao, Terence},
  author={Ziegler, Tamar},
  title={The inverse conjecture for the Gowers norm over finite fields via the correspondence principle},
  journal={Anal. PDE},
  volume={3},
  year={2010},
  number={1},
  pages={1--20},
}

\bib{tz2}{article}{
  author={Tao, Terence},
  author={Ziegler, Tamar},
  title={The inverse conjecture for the Gowers norm over finite fields in low characteristic},
  journal={Ann. Comb.},
  volume={16},
  year={2012},
  number={1},
  pages={121--188},
}

\bib{wilson}{article}{
  author={Wilson, John},
  title={First-order group theory},
  conference={
    title={Infinite groups 1994 (Ravello)},
  },
  book={
    publisher={de Gruyter, Berlin},
  },
  date={1996},
  pages={301--314},
}

\end{biblist}
\end{bibdiv}
\end{document}